\documentclass[10pt]{amsart} 

\usepackage[utf8]{inputenc} 

\usepackage{lipsum}

\usepackage{geometry} 
\usepackage{graphicx} 

\usepackage[english]{babel}

\usepackage{booktabs} 
\usepackage{array} 
\usepackage{paralist} 
\usepackage{verbatim} 
\usepackage{subfig} 

\usepackage{caption}

\usepackage{amsfonts,amsmath,amssymb,bbm,amsthm,amsbsy}

\usepackage{mathtools}

\makeatletter							
\newtheorem*{rep@theorem}{\rep@title}
\newcommand{\newreptheorem}[2]{%
\newenvironment{rep#1}[1]{%
 \def\rep@title{#2 \ref{##1}}%
 \begin{rep@theorem}}%
 {\end{rep@theorem}}}
\makeatother

\newtheorem{thm}{Theorem}[section]
\newtheorem{lemma}[thm]{Lemma}
\newtheorem{prop}[thm]{Proposition}

\newtheorem{claim}[thm]{Claim}

\newtheorem*{thm*}{Theorem}
\newtheorem*{lemma*}{Lemma}
\newtheorem*{prop*}{Proposition}
\newtheorem*{corr*}{Corrolary}
\newtheorem*{claim*}{Claim}

\theoremstyle{remark}

\newtheorem{quest}[thm]{Question}

\newtheorem*{rmk*}{Remark}
\newtheorem*{conj*}{Conjecture}
\newtheorem*{quest*}{Question}

\theoremstyle{definition}
\newtheorem{defn}[thm]{Definition}
\newtheorem{exmp}[thm]{Example}

\newtheorem*{defn*}{Definition}
\newtheorem*{exmp*}{Example}

\newtheorem{theorem}{Theorem}

\newreptheorem{theorem}{Theorem}
\newreptheorem{corollary}{Corollary}
\newreptheorem{proposition}{Proposition}

\usepackage{fancyhdr} 
\usepackage{tikz}
\usepackage{tikz-cd}

\usepackage{IEEEtrantools}

\newenvironment{equ*}[1]{\begin{IEEEeqnarray*}{#1}}{\end{IEEEeqnarray*}}

\newcommand{\R}{\mathbb{R}}

\newcommand{\Z}{\mathbb{Z}}
\newcommand{\N}{\mathbb{N}}

\newcommand{\inj}{\hookrightarrow}

\newcommand{\col}{\colon}

\newcommand{\att}{\texttt{a}}
\newcommand{\btt}{\texttt{b}}

\newcommand{\dtt}{\texttt{d}}
\newcommand{\xtt}{\texttt{x}}
\newcommand{\ytt}{\texttt{y}}

\newcommand{\Hrm}{\mathrm{H}}

\usepackage{pdfpages}
\usepackage{xspace}

\author{Nicolaus Heuer}
\title{Cup Product in Bounded Cohomology of the Free Group}
\address{Department of Mathematics\\
  University of Oxford}
\email[N.~Heuer]{heuer@maths.ox.ac.uk}
\date{\today}

\begin{document}

\begin{abstract}
The theory of bounded cohomology of groups has many applications.
A key open problem is to compute the full bounded cohomology $\Hrm_b^n(F, \R)$ of a non-abelian free group $F$ with trivial real coefficients.
It is known that $\Hrm_b^n(F,\R)$ is trivial for $n=1$ and uncountable dimensional for $n=2,3$, but $\Hrm_b^n(F,\R)$ remains unknown for any $n \geq 4$.
For $n=4$, one may construct classes by taking the cup product $\alpha \smile \beta \in \Hrm_b^4(F, \R)$ between two $2$-classes $\alpha, \beta \in \Hrm^2_b(F, \R)$. 
However, we show that all such cup products are trivial if $\alpha$ and $\beta$ are classes induced by the quasimorphisms defined by Brooks or Rolli.
\end{abstract}

\maketitle

\section{Introduction}

Bounded cohomology of groups was originally studied by Gromov in \cite{Gromov}. Since then bounded cohomology emerged as an indepenedent research field with many applications. These include stable commutator length (\cite{calegari:scl}), circle actions (\cite{circle-actions}) and the Chern Conjecture. See \cite{monod_invitation} for a survey and \cite{frigerio} for a book on bounded cohomology.

However, the bounded cohomology of a group $G$ is notoriously hard to explicitly compute, even for trivial real coefficients.
On the one hand, it is known that $\Hrm_b^n(G,\R)$ is trivial for all $n \in \N$ if $G$ is amenable.
On the other hand, $\Hrm_b^n(G,\R)$ is uncountable dimensional if $G$ is an acylindrically hyperbolic group and $n=2$ or $n=3$; see \cite{hull_osin}, \cite{soma} and \cite{fps}. 

This paper will exclusively focus on the bounded cohomology of non-abelian free groups $F$ with trivial real coefficients, denoted by $\Hrm_b^n(F, \R)$.
We note that $\Hrm_b^n(F,\R)$ is fully unknown for any $n \geq 4$.
Free groups play a distinguished r\^ole  in constructing non-trivial classes on other acylindrically hyperbolic groups. Due to a result by Frigerio, Pozzetti and Sisto, any non-trivial alternating class in $\Hrm_b^n(F, \R)$ may be promoted to a non-trivial class in $\Hrm_b^n(G,\R)$ where $G$ is an acylindrically hyperbolic group and $n \geq 2$; see Corollary 2 of \cite{fps}.

All classes in the second bounded cohomology of a non-abelian free group $F$ with trivial real coefficients arise as coboundaries of \emph{quasimorphisms} (see Subsection \ref{subsec:bounded cohomology}) i.e.\ for any $\omega \in \Hrm_b^2(F,\R)$ there is a quasimorphism $\phi \col F \to \R$ such that $[\delta^1 \phi] = \omega$.
There are many explicit constructions of quasimorphisms $\phi \col F \to \R$, most prominently the one defined by Brooks \cite{Brooks} and Rolli \cite{rolli}; see Subsection \ref{subsec:bounded cohomology}.
One may hope to construct non-trivial classes in $\Hrm^4_b(F, \R)$ by taking the cup product
$[\delta^1 \phi] \smile [\delta^1 \psi] \in \Hrm^4_b(F, \R)$ between two such quasimorphisms $\phi, \psi \col F \to \R$.
We will show that this approach fails.
\begin{theorem} \label{theorem:brooks and rolli}
Let $\phi, \psi \col F \to \R$ be two quasimorphisms on a non-abelian free group $F$ where each of $\phi$ and $\psi$ is either Brooks counting quasimorphisms on a non self-overlapping word or quasimorphisms in the sense of Rolli. 
Then $[\delta^1 \phi] \smile [\delta^1 \psi] \in \Hrm^4_b(F,\R)$ is trivial.
\end{theorem}

 We note that Michelle Bucher and Nicolas Monod have independently proved the vanishing of the cup product between the classes induced by Brooks quasimorphisms with a different technique; see \cite{BucherMonod}.
 
Theorem \ref{theorem:brooks and rolli} will follow from a more general vanishing Theorem. 
For this, we will first define \emph{decompositions} (see Definition \ref{def:decomposition}) which are certain maps $\Delta$ that assign to each element $g \in F$ a finite sequence $(g_1, \ldots, g_n)$ of arbitrary length with $g_j \in F$ and such that $g = g_1 \cdots g_n$ and there is no cancellation between the $g_j$.
We then define two new classes of quasimorphisms, namely \emph{$\Delta$-decomposable quasimorphisms}  (Definition \ref{defn:decomposable quasimorphism}) and \emph{$\Delta$-continuous quasimorphisms} (Definition \ref{def:continuous}).
Each Brooks and Rolli quasimorphism will be both $\Delta$-decomposable and $\Delta$-continuous with respect to some decomposition $\Delta$.
We will show:

\begin{theorem} \label{thm:main}
Let $\Delta$ be a decomposition of $F$ and let $\phi, \psi \col F \to \R$ be quasimorphisms such that $\phi$ is $\Delta$-decomposable and $\psi$ is $\Delta$-continuous. Then $[\delta^1 \phi] \smile [\delta^1 \psi]  \in \Hrm^4_b(F, \R)$ is trivial.
\end{theorem}

We will prove Theorem \ref{thm:main} by giving an explicit bounded coboundary in terms of $\phi$ and $\psi$ in Theorem \ref{thm:technical}.
Let $\phi$ and $\psi$ be as in Theorem \ref{thm:main}.
A key observation of this paper is that the function $(g,h,i) \mapsto \phi(g)\delta^1 \psi(h,i)$ 
 ``behaves like a honest cocycle with respect to $\Delta$''. The idea of the proof of Theorem \ref{thm:technical} is to mimic the algebraic proof that 
honest cocycles on free groups have a coboundary; see Subsection \ref{subsec:idea of technical theorem}.

It was shown by Grigorchuk \cite{grigorchuk} that Brooks quasimorphisms are \emph{dense} in the vector space of quasimorphisms in the topology of pointwise convergence. 
In light of Theorem \ref{theorem:brooks and rolli} one would like to deduce from this density that the cup product
between all bounded $2$-classes vanishes. However, this does not seem straightforward.
The topology needed for such a deduction is the stronger \emph{defect topology}.
Brooks cocycles are not dense in this topology, in fact the space of $2$-cocycles is not even separable in this topology.
We therefore ask:
\begin{quest}
Let $F$ be a non-abelian free group. Is the cup product 
$$
\smile \col \Hrm^2_b(F,\R) \times \Hrm^2_b(F,\R) \to \Hrm^4_b(F,\R)
$$
trivial?
\end{quest}
Note that it is unknown if nontrivial classes in $\Hrm^4_b(F, \R)$ exist. 
We mention that the cup product on bounded cohomology for other groups need not be trivial. Let $G$ be a group with non-trivial second bounded cohomology. Then $G \times G$ admits a non-trivial cup product $$
\smile \col \Hrm_b^2(G \times G, \R) \times \Hrm_b^2(G \times G, \R) \to \Hrm^4_b(G \times G, \R)
$$
induced by the factors. See \cite{claraloeh} for results and constructions in bounded cohomology using the cup product.

\subsection*{Organisation}
This paper is organised as follows: Section \ref{sec:preliminaries} introduces notation and recalls basic facts about bounded cohomology. 
Section \ref{sec:decomposition} defines and studies \emph{decompositions} $\Delta$ of non-abelian free groups $F$ mentioned above as well as $\Delta$-decomposable and $\Delta$-continuous quasimorphisms.
In Section \ref{sec:the technical theorem} we will introduce and prove Theorem \ref{thm:technical}, which will provide the explicit bounded primitives for the cup products studied in this paper. The key ideas of the proof are illustrated in Subsection \ref{subsec:idea of technical theorem}.
Theorems \ref{theorem:brooks and rolli} and \ref{thm:main} will be corollaries of Theorem \ref{thm:technical} and proved in Subsection \ref{subsec:theorems a and b}.

\subsection*{Acknowledgements}

I would like to thank  Roberto Frigerio, Clara L\"oh, Michelle Bucher and Marco Moraschini for helpful discussions and detailed comments and my supervisor Martin Bridson 
for his helpful comments and support.
I would further like to thank the anonymous referee for many helpful
remarks which substantially improved the paper.
The author would like to thank the Isaac Newton Institute for Mathematical Sciences, Cambridge, for support
and hospitality during the programme \textit{Non-Positive Curvature Group
Actions and Cohomology} where work on this paper was undertaken. 
This work was supported by EPSRC grant no EP/K032208/1. The author is also supported by the Oxford-Cocker Scholarship.

\newpage

\section{Preliminaries} \label{sec:preliminaries}

In Subsection \ref{subsec:Notation and conventions} we will introduce notations which will be used throughout the paper.
In Subsection \ref{subsec:bounded cohomology} we define (bounded) cohomology of groups and the cup product. In Subsection \ref{subsec:quasimorphisms} we will define quasimorphisms, in particular the quasimorphisms defined by Brooks and Rolli.

\subsection{Notation and conventions} \label{subsec:Notation and conventions}

A generic group will be denoted by $G$ and a non-abelian free group will be denoted by $F$.
The generating set $\mathcal{S}$ of $F$ will consist of letters in code-font (''$\att,\btt$``). The identity element of a group will be denoted by ''$1$``. Small Roman letters (''$a,b$``) typically denote elements of groups. Curly 
capitals (''$\mathcal{A}, \mathcal{B}$``) denote sets, typically subsets of $F$. Functions (typically from $F^k$ to $\R$) will be denoted by Greek letters (''$\alpha, \beta$``).
We stick to this notation unless it is mathematical convention to do otherwise.

\subsection{(Bounded) cohomology and the cup product} \label{subsec:bounded cohomology}
We recall the definition of discrete (bounded) cohomology with trivial real coefficients using the inhomogeneous resolution.
Let $G$ be a group, let $C^k(G, \R) = \{ \phi \col G^k \to \R \}$ and let $C^k_b(G, \R) \subset C^k(G, \R)$ be the subset of bounded functions with respect to the supremum norm.  We define the coboundary operator $\delta^n \col C^n(G,\R) \to C^{n+1}(G,\R)$ via
\begin{equ*}{rcl}
 \delta^n(\alpha)(g_1, \ldots, g_{n+1}) &=& \alpha(g_2, \ldots, g_{n+1}) \\
 & &+ \sum_{j=1}^n (-1)^j \alpha(g_1, \ldots, g_j g_{j+1}, \ldots, g_{n+1}) \\
 & &+(-1)^{n+1} \alpha(g_1, \ldots, g_n)
 \end{equ*}
 and note that it restricts to $\delta^n \col C^n_b(G, \R) \to C^{n+1}_b(G,\R)$.
The homology of the cochain complex $(C^*(G,\R), \delta^*)$ is called the \emph{cohomology of the group $G$ with trivial real coefficients}
and denoted by $\Hrm^*(G,\R)$.
Similarly the homology of $(C^*_b(G,\R), \delta^*)$  is called the the \emph{bounded cohomology of $G$ with trivial real coefficients} and 
denote it by $\Hrm^*_b(G,\R)$.
The inclusion $C^n_b(G, \R) \inj C^n(G, \R)$ is a chain map which induces a map $c^n \col \Hrm^n_b(G,\R) \to \Hrm^n(G,\R)$ called the \emph{comparison map}. Cocycles in the kernel 
of $c^n$ are called \emph{exact} classes and will correspond to quasimorphisms if $n=2$; see Subsection \ref{subsec:quasimorphisms}.
For a detailed discussion see \cite{monod_invitation} for a survey and \cite{frigerio} for a book on bounded cohomology.

The \emph{cup product} is a map $\smile \col \Hrm^n(G,\R) \times \Hrm^m(G, \R) \to \Hrm^{n+m}(G, \R)$ defined by setting $([\omega_1], [\omega_2]) \mapsto [\omega_1] \smile [\omega_2]$ where
$[\omega_1] \smile [\omega_2] \in \Hrm^{n+m}(G,\R)$ is represented by the cocycle $\omega_1 \smile \omega_2 \in C^{n+m}(G,\R)$ defined via
\[
\omega_1 \smile \omega_2 \col (g_1, \ldots, g_n, g_{n+1}, \ldots, g_{n+m}) \mapsto \omega_1(g_1, \ldots, g_n) \cdot \omega_2(g_{n+1}, \ldots, g_{n+m}).
\]
It is easy to check that this map induces a well-defined map 
$$
\smile \col \Hrm^n_b(G,\R) \times \Hrm^m_b(G, \R) \to \Hrm^{n+m}_b(G, \R).
$$

\subsection{Quasimorphisms} \label{subsec:quasimorphisms}
 
A \emph{quasimorphism} is a map $\alpha \col G \to \R$ such that there is a constant $D>0$ such that for every $g,h \in G$, $|\alpha(g) - \alpha(gh) + \alpha(h)| < D$ and hence $\delta^1 \alpha \in C^2_b(G, \R)$.
Hence the exact $2$-classes of $\Hrm^2_b(G, \R)$ are exactly the coboundaries of 
quasimorphisms.
A quasimorphism $\alpha \col G \to \R$ will be called \emph{symmetric} if $\alpha$ satisfies in addition that
$\alpha(g) = -\alpha(g^{-1})$ for all $g \in G$. In this case, we call its coboundary $\delta^1 \alpha \in C^2_b(G, \R)$
symmetric as well. It is easy to see that each exact $2$-class is represented by a symmetric cocycle.
On a non-abelian free group $F$ there are several constructions of non-trivial quasimorphisms.

\begin{exmp} \label{exmp:brooks quasimorphism}
In \cite{Brooks}, Brooks gave the first example of an infinite family of linearly independent quasimorphisms on the free group. Let $F$ be a non-abelian free group on a fixed generating set $\mathcal{S}$. Let $w, g \in F$ be two elements which are represented by reduced words $w = \ytt_1 \cdots \ytt_n$ and $g = \xtt_1 \cdots \xtt_m$, where $\xtt_j, \ytt_j$ are letters of $F$.
We say that $w$ is a sub-word of $g$ if $n \leq m$ and there is an $s \in \{0, \ldots, m-n \}$ such that $\ytt_j = \xtt_{j+s}$ for all $j = 1, \ldots, n$.
 Let $w$ be a reduced \emph{non self-overlapping} word, i.e.\ a word $w$ such that there are no words $x$ and $y$ with $x$ non-trivial such that $w = x y x$ as a reduced word.
For $w$ a non self-overlapping word we define the function $\nu_w \col F \to \Z$ by setting $\nu_w \col g  \mapsto  \# \{ \mbox{$w$ is a subword of $g$} \}$. Then the \emph{Brooks counting quasimorphism on the word $w$} is the function 
$$
\phi_w = \nu_w - \nu_{w^{-1}}.
$$ It is easy to see that this defines a symmetric quasimorphism.
\end{exmp}

\begin{exmp} \label{exmp:rolli quasimorphism}
In \cite{rolli}, Rolli gave a different example of an infinite family of linearly independent quasimorphisms.
Suppose $F$ is generated by $\mathcal{S} = \{ \xtt_1, \ldots, \xtt_n \}$.
Let $\lambda_1, \ldots, \lambda_n \in \ell^\infty_{alt}(\Z)$ be bounded functions $\lambda_j \col \Z \to \R$ that satisfy $\lambda_j(-n) = -\lambda_j(n)$.
Each non-trivial element $g \in F$ may be uniquely written as $g = \xtt_{n_1}^{m_1} \cdots \xtt_{n_k}^{m_k}$ where all $m_j$ are non-zero and no consecutive $n_j$ are the same. 
Then we can see that the map $\phi \col F \to \R$ defined by setting
$$
\phi \col g \mapsto \sum_{j=1}^k \lambda_{n_j}(m_j)
$$
is a symmetric quasimorphism called \emph{Rolli quasimorphism}.
\end{exmp}

We will generalise Brooks and Rolli quasimorphisms in the next section.

\section{Decomposition} \label{sec:decomposition}

The aim of this section is to introduce \emph{decompositions of $F$} in Subsection \ref{subsec:decompositions of F}.
Let $F$ be a non-abelian free group with a fixed set of generators.
Crudely, a decomposition $\Delta$ is a way of assigning a finite sequence $(g_1, \ldots g_k)$ of elements $g_j \in F$ to an element $g \in F$ such that $g= g_1 \cdots g_k$ as a reduced word and such that that this decomposition behaves well on geodesic triangles in the Cayley graph.
We will see that any decomposition $\Delta$ induces a quasimorphism (Proposition \ref{prop: decompositions define quasimorphisms}), called $\Delta$-decomposable quasimorphism in Subsection \ref{subsec:decomposable quasimorphisms}.
We will introduce special decompositions, $\Delta_{triv}$, $\Delta_w$ and $\Delta_{rolli}$ and see that
$\Delta_{triv}$-decomposable quasimorphisms are exactly the homomorphisms $F \to \R$, that Brooks quasimorphisms on a non self-overlapping word $w$ are $\Delta_w$-decomposable and that the quasimorphisms in the sense of Rolli are $\Delta_{rolli}$-decomposable.
In Subsection \ref{subsec:continuous quasimorphisms} we introduce $\Delta$-continuous cocycles.

\subsection{Notation for sequences} \label{subsec:notation for sequences}
A set $\mathcal{A} \subset F$ will be called \emph{symmetric} if $a \in \mathcal{A}$ implies that $a^{-1} \in \mathcal{A}$. For such a symmetric 
set $\mathcal{A} \subset F$, we denote by $\mathcal{A}^*$ the set of finite sequences in $\mathcal{A}$ including the empty sequence. This is, the set 
of all expressions $(a_1, \ldots, a_k)$ where $k \in \N^0$ is arbitrary and $a_j \in \mathcal{A}$. We will denote the element $(a_1, \ldots, a_k) \in \mathcal{A}^*$  
by $(\underline{a})$ and $k$ will be  called the \emph{length} of $(\underline{a})$ where we set $k=0$ if $(\underline{a})$ is the empty sequence.  
For a sequence $(\underline{a})$, we denote by $(\underline{a}^{-1})$ the sequence $(a_k^{-1}, \ldots, a_1^{-1}) \in \mathcal{A}^*$ and 
the element $\bar{a} \in F$ denotes the product $a_1 \cdots a_k \in F$.
We will often work with multi-indexes: The sequences $(\underline{a}_1), (\underline{a}_2), (\underline{a}_3) \in \mathcal{A}^*$ will correspond to the sequences $(\underline{a}_j) = (a_{j,1}, \ldots, a_{j,n_j})$, 
where $n_j$ is the length of $(\underline{a}_j)$ for $j=1,2,3$.
For two sequences $(\underline{a}) = (a_1, \ldots, a_k)$ and $(\underline{b}) = (b_1, \ldots, b_l)$ we define the \emph{common sequence of $(\underline{a})$ and $(\underline{b})$} to be the empty sequence if $a_1 \not = b_1$ and
to be the sequence $(\underline{c}) = (a_1, \ldots, a_n)$ where $n$ is the largest integer with $n \leq \min \{k,l \}$ such that 
$a_j=b_j$ for all $j \leq n$. Moreover, $(\underline{a}) \cdot (\underline{b})$ will denote the sequence $(a_1, \ldots, a_k, b_1, \ldots, b_l)$.

\subsection{Decompositions of $F$} \label{subsec:decompositions of F}
We now define the main tool of this paper, namely \emph{decompositions}. As mentioned in the introduction we will restrict our attention to non-abelian free groups $F$ on a fixed generating set $\mathcal{S}$. 
\begin{defn} \label{def:decomposition}
Let $\mathcal{P} \subset F$ be a symmetric set of elements of $F$ called \emph{pieces} and assume that $\mathcal{P}$ does not contain the identity.
A \emph{decomposition of $F$ into the pieces $\mathcal{P}$} is a map $\Delta \col F \to \mathcal{P}^*$ assigning to every element $g \in F$ a finite sequence
$\Delta(g) = (g_1, \ldots, g_k)$ with $g_j \in \mathcal{P}$ such that:
\begin{enumerate}
\item For every $g \in F$ and $\Delta(g) = (g_1, \ldots, g_k)$ we have $g = g_1 \cdots g_k$ as a reduced word (no cancelation). Also, we require that  $\Delta(g^{-1}) = (g_k^{-1}, \ldots, g_1^{-1})$.
\item \label{property:suffix_closed} For every $g \in F$ with $\Delta(g) = (g_1, \ldots, g_k)$ we have $\Delta(g_i \cdots g_j) = (g_i, \ldots, g_j)$ 
for $1 \leq i \leq j \leq k$. We refer to this property as $\Delta$ being \emph{infix closed}.
\item 
There is a constant $R > 0$ with the following property.

Let $g,h \in F$ and let 
\begin{itemize}
\item $(\underline{c}_1) \in \mathcal{P}^* $ be such that $(\underline{c}_1^{-1})$ is the common sequence of $\Delta(g)$ and $\Delta(gh)$, 
\item $(\underline{c}_2) \in \mathcal{P}^* $ be such that $(\underline{c}_2^{-1})$ is the common sequence of $\Delta(g^{-1})$ and $\Delta(h)$ and 
\item $(\underline{c}_3) \in \mathcal{P}^* $ be such that $(\underline{c}_3^{-1})$ is the common sequence of $\Delta(h^{-1})$ and $\Delta(h^{-1} g^{-1})$. 
\end{itemize}
It is not difficult to see that there are $(\underline{r}_1), (\underline{r}_2), (\underline{r}_3) \in \mathcal{P}^*$ such that
\begin{align*}
\Delta(g) &= ( \underline{c}_1^{-1}) \cdot (\underline{r}_1) \cdot (\underline{c}_2) \\
\Delta(h) &= ( \underline{c}_2^{-1} ) \cdot (\underline{r}_2) \cdot (\underline{c}_3) \mbox{ and }\\
\Delta(h^{-1} g^{-1}) &= ( \underline{c}_3^{-1}) \cdot (\underline{r}_3) \cdot  (\underline{c}_1).
\end{align*}
Then the length of $(\underline{r}_1)$, $(\underline{r}_2)$ and $(\underline{r}_3)$ is bounded by $R$.
See Figure \ref{fig:Figure1} for a geometric interpretation and Subsection \ref{subsec:notation for sequences} for the notation of common sequences and concatenation of sequences.
\end{enumerate}
For such a pair $(g,h)$ we will call $(\underline{c}_1), (\underline{c}_2),(\underline{c}_3)$ the \emph{ $c$-part of the $\Delta$-triangle of $(g,h)$} and 
$(\underline{r}_1), (\underline{r}_2),(\underline{r}_3)$ the  \emph{$r$-part of the $\Delta$-triangle of $(g,h)$}.
A sequence $(g_1, \ldots, g_k)$ such that 
$$
\Delta(g_1 \cdots g_k) = (g_1, \ldots, g_k)
$$
will be called a \emph{proper $\Delta$ sequence}. 
\end{defn}

 \begin{figure}
  \centering

  \subfloat{\includegraphics[width=0.3\textwidth]{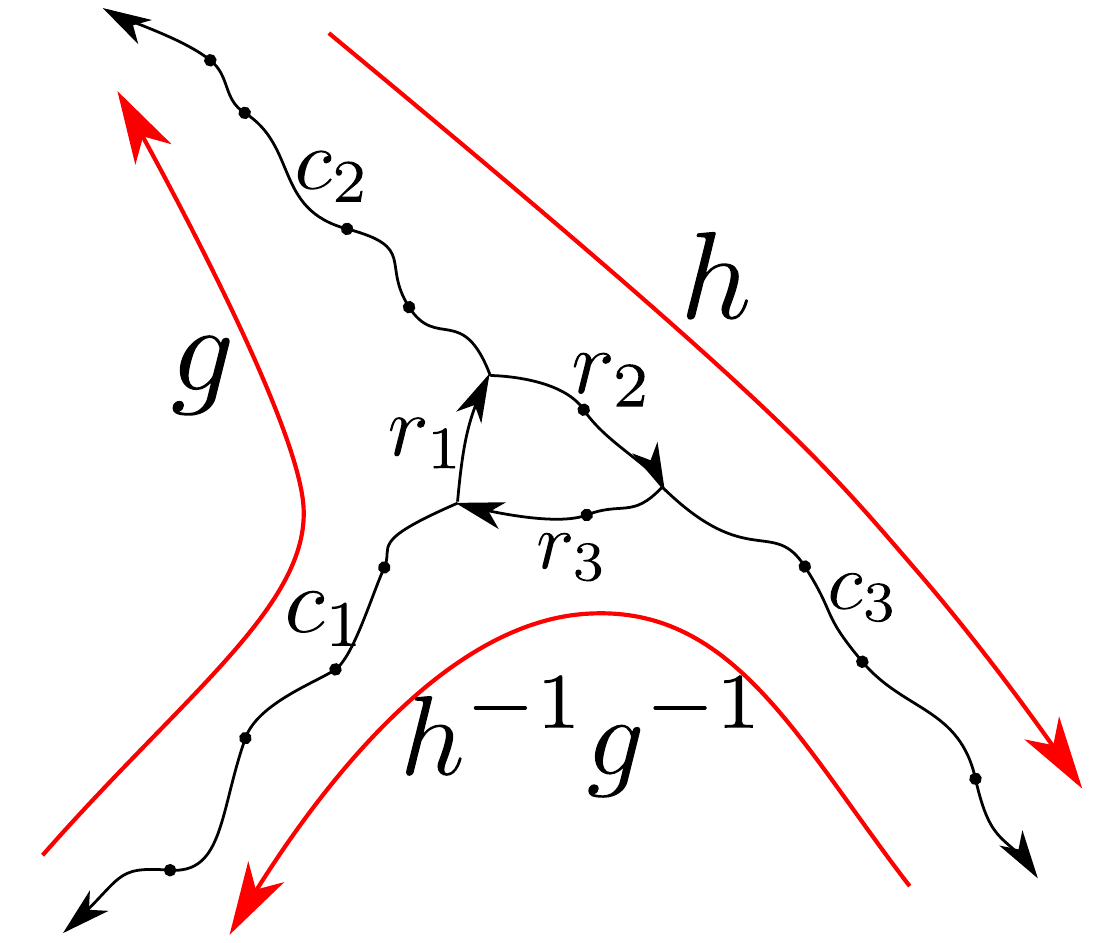}}

  \caption{$\Delta(g)$, $\Delta(h)$ and $\Delta(h^{-1} g^{-1})$ have sides which can be identified.} \label{fig:Figure1}
\end{figure}

\begin{exmp} \label{exmp:trivial decomposition}
Let $\mathcal{S} = \{ \xtt_1, \ldots, \xtt_n \}$ be an alphabet generating $F$. Every word $w \in F$ may be uniquely written as a  word
$w = \ytt_1 \cdots \ytt_k$
without backtracking where $\ytt_i \in \mathcal{S}^\pm$.
Set $\mathcal{P}_{triv} = \mathcal{S}^\pm$ and define the map
$\Delta_{triv} \col F \to \mathcal{P}_{triv}^*$ by setting
$$
\Delta_{triv} \col w \mapsto (\ytt_1, \ldots, \ytt_k)
$$
for $w$ as above.
Then we see that $\Delta_{triv}$ is indeed a decomposition.
Let $g,h \in G$ and let $c_1, c_2, c_3$ be such that
$g = c_1^{-1} c_2$, $h = c_2^{-1} c_3$ and $gh = c_1^{-1} c_3$ as reduced words.
Then the $c$-part of the $\Delta_{triv}$-triangle of $(g,h)$ is
$\Delta_{triv}(c_1), \Delta_{triv}(c_2), \Delta_{triv}(c_3)$ and the $r$-part of the $\Delta_{triv}$-triangle of $(g,h)$ is $(\emptyset), (\emptyset), (\emptyset)$ where $(\emptyset)$ denotes the empty sequence.

We call the map $\Delta_{triv}$ the \emph{trivial decomposition}.
\end{exmp}

\begin{exmp} \label{exmp:brooks decomposition}
Let $w \in F$ be a non self-overlapping word (see Example \ref{exmp:brooks quasimorphism}). Every word $g \in F$ may be written as $g = u_1 w^{\epsilon_1} u_2 \cdots u_{k-1} w^{\epsilon_{k-1}} u_k$, where the $u_j$'s may be empty, $\epsilon_j \in \{-1, +1 \}$ and no $u_j$ contains $w$ or $w^{-1}$ as subwords. It is not hard to show that this expression is unique. Observe that a reduced word in the free group does not overlap with its inverse.
Set $\mathcal{P}_w = \{ u \in F \mid \mbox{neither $w$ nor $w^{-1}$ are subwords of $u$} \} \cup \{ w, w^{-1} \}$.

We define the \emph{Brooks-decomposition on the word $w$} as the map $\Delta_w \col F \to \mathcal{P}_w^*$ by setting
$$
\Delta_w \col g \to (u_1, w^{\epsilon_1}, u_2, \cdots, u_{k-1}, w^{\epsilon_{k-1}},u_k)
$$
for $g$ as above.
It is easy to check that this is indeed a decomposition.
\end{exmp}

\begin{exmp} \label{exmp:rolli decomposition}
As in Example \ref{exmp:rolli quasimorphism}, suppose that $F$ is generated by $\mathcal{S} = \{ \xtt_1, \ldots, \xtt_n \}$ and observe that every non-trivial element $g \in F$ may be uniquely written as $g = \xtt_{n_1}^{m_1} \cdots \xtt_{n_k}^{m_k}$ where all $m_j$ are non-zero and no consecutive $n_j$ are the same. 
Set $\mathcal{P}_{rolli} = \{ \xtt_j^m \mid j \in \{ 1, \ldots, n \}, m \in \Z \}$.
We define the \emph{Rolli-decompostion} as the map $\Delta_{rolli} \col F \to \mathcal{P}_{rolli}^*$ via
$$
\Delta_{rolli} \col g \mapsto (\xtt_{n_1}^{m_1}, \ldots , \xtt_{n_k}^{m_k})
$$
for $g$ as above. It is easy to check that this is indeed a decomposition.
\end{exmp}

Often we just talk about the decomposition without specifying the pieces $\mathcal{P}$ explicitly.
From a decomposition $\Delta$ we derive the notion of two sorts of quasimorphisms: \emph{$\Delta$-decomposable quasimorphisms} (Definition \ref{defn:decomposable quasimorphism}) and \emph{$\Delta$-continuous quasimorphisms} (Definition \ref{def:continuous}).

\subsection{$\Delta$-decomposable quasimorphisms} \label{subsec:decomposable quasimorphisms}
Each decomposition $\Delta$ of $F$ induces many different quasimorphisms on $F$.
\begin{defn} \label{defn:decomposable quasimorphism}
Let $\Delta$ be a decomposition with pieces $\mathcal{P}$ and let $\lambda \in \ell^\infty_{alt}(\mathcal{P})$ be a symmetric bounded map on $\mathcal{P}$, i.e. $\lambda(p^{-1}) = - \lambda(p)$ for every $p \in \mathcal{P}$.
Then the map $\phi_{\lambda, \Delta} \col F \to \R$ defined via
$$
\phi_{\lambda, \Delta} \col g \mapsto \sum_{j=1}^k \lambda(g_j)
$$
where $\Delta(g) = (g_1, \ldots, g_k)$
is called a \emph{$\Delta$-decomposable} quasimorphism.
\end{defn}

We may check that such a $\phi_{\lambda, \Delta}$ is indeed a quasimorphism.
\begin{prop} \label{prop: decompositions define quasimorphisms}
Let $\Delta$ and $\lambda$ be as in Definition \ref{defn:decomposable quasimorphism}. Then $\phi_{\lambda, \Delta}$ is a symmetric quasimorphism.
If $g, g' \in F$ are such that $\Delta(g \cdot g') = (\Delta(g)) \cdot (\Delta(g'))$ then $\delta^1 \phi(g, g') = 0$.
In particular, for all $g \in G$ with $\Delta(g) = (g_1, \ldots, g_k)$ we have that $\delta^1 \phi_{\lambda, \Delta}(g_j, g_{j+1} \cdots g_k) = 0$ for $j = 1, \ldots, k-1$.
\end{prop}

\begin{proof}
Symmetry is immediate from the assumptions on $\Delta(g^{-1})$ and $\lambda$. 
Let $g,h \in F$ and let $(\underline{c}_j), (\underline{r}_j)$, $j \in \{ 1,2,3 \}$ be as in the definition of the decomposition.
We compute
\begin{align*}
\phi_{\lambda, \Delta}(g) = - \sum_{j=1}^{n_1} \lambda(c_{1,j}) + \sum_{j=1}^{m_1} \lambda(r_{1,j}) + \sum_{j=1}^{n_2} \lambda(c_{2,j}) \\
\phi_{\lambda, \Delta}(h) = - \sum_{j=1}^{n_2} \lambda(c_{2,j}) + \sum_{j=1}^{m_2} \lambda(r_{2,j}) + \sum_{j=1}^{n_3} \lambda(c_{3,j}) \\
\phi_{\lambda, \Delta}(gh) = - \sum_{j=1}^{n_1} \lambda(c_{1,j}) - \sum_{j=1}^{m_3} \lambda(r_{3,j}) + \sum_{j=1}^{n_3} \lambda(c_{3,j})
\end{align*}
and hence
\[
\delta^1 \phi_{\lambda, \Delta}(g,h) = \phi_{\lambda, \Delta}(g) + \phi_{\lambda, \Delta}(h) - \phi_{\lambda, \Delta}(gh) =  \sum_{j=1}^{m_1} \lambda(r_{1,j}) + \sum_{j=1}^{m_2} \lambda(r_{2,j}) + \sum_{j=1}^{m_3} \lambda(r_{3,j})
\]
and hence $| \delta^1 \phi_{\lambda, \Delta}(g,h)| \leq 3 R \| \lambda \|_\infty$.
Note that from this calculation we also see that $\delta^1 \phi_{\lambda, \Delta}(g,h)$ only depends on the $r$-part of the $\Delta$-triangle for $(g,h)$ and not on the $c$-part. 
The second part follows immediately from property (\ref{property:suffix_closed}) of a decomposition.
\end{proof}

Both Brooks and Rolli quasimorphisms are $\Delta$-decomposable quasimorphisms with respect to some $\Delta$ as the following examples show: 

\begin{exmp}
Let $\Delta_{triv}$ be the trivial decomposition of
Example \ref{exmp:trivial decomposition}.
It is easy to see that the $\Delta_{triv}$-decomposable quasimorphisms are exactly the homomorphisms $\phi \col F \to \R$.
\end{exmp}

\begin{exmp}
Let $\mathcal{P}_w$ be as in Example \ref{exmp:brooks decomposition} and define $\lambda \col \mathcal{P}_w \to \R$ by setting
$$
\lambda \col p \mapsto
\begin{cases}
1 & \mbox{if } p = w,\\
-1 & \mbox{if } p = w^{-1},\\
0 & \mbox{otherwise}.
\end{cases}
$$ 
Then it we see that the induced decomposable quasimorphism $\phi_{\lambda, \Delta_w}$ is the Brooks counting quasimorphism on $w$; see Example \ref{exmp:brooks quasimorphism}.
\end{exmp}

\begin{exmp}
Let $\lambda_1, \ldots, \lambda_n$ be as in Example \ref{exmp:rolli quasimorphism} and let $\mathcal{P}_{rolli}$ be as in Example \ref{exmp:rolli decomposition}. Define $\lambda \col \mathcal{P}_{rolli} \mapsto \R$ by setting
$$
\lambda \col x_j^m \mapsto \lambda_j(m).
$$
Then we see that the induced quasimorphism $\phi_{\lambda, \Delta_{rolli}}$ is a Rolli quasimorphism; see Example \ref{exmp:rolli quasimorphism}.
\end{exmp}

\subsection{$\Delta$-continuous quasimorphisms and cocycles} \label{subsec:continuous quasimorphisms}

We will define $\Delta$-continuous cocycles.
Crudely, a cocycle $\omega$ is $\Delta$-continuous, if the value $\omega(g,h)$ depends ``mostly'' on the neighbourhood of the midpoint of the geodesic triangle spanned by $e, g, gh$ in the Cayley graph of $F$.
For this, we will first establish a notion of when two pairs $(g,h)$ and $(g',h')$ of elements in $F$ define triangles which are ``close''.

For this we define the function $N_{\Delta} \col F^2 \times F^2 \to \N \cup \infty$ as follows.
Let $(g,h) \in F^2$ and $(g', h') \in F^2$ be two pairs of elements of $F$.
Let $(\underline{c}_j), (\underline{r}_j)$
for $j=1,2,3$ be the $\Delta$-triangle of $(g, h)$
where $(\underline{c}_j)$ has length $n_j$
 and let $(\underline{c}'_j), (\underline{r}'_j)$
for $j=1,2,3$ be the $\Delta$-triangle of $(g', h')$ where $(\underline{c}'_j)$ has length $n'_j$.

We set $N_\Delta((g,h),(g',h')) = 0$ if there is a $j \in \{1,2,3 \}$ such that $r_j \not = r'_j$ and
$N_\Delta((g,h),(g',h'))=\infty$ if $(g,h) = (g',h')$.
Else, let $N_\Delta((g,h),(g',h'))$ be the largest integer $N$ which satisfies that $N \leq \min \{n_j, n_j' \}$ and $c_{j,k}=c'_{j,k}$ for every $k \leq N$ and $j \in \{1,2,3 \}$ such that $c_j \not = c'_j$.

Observe that $N_\Delta((g,h), (g',h')) = \infty$ if and only if $(g,h) = (g',h')$.
This is because if $(g,h) \not = (g',h')$ then either there is some $j$ such that $r_j \not = r'_j$, in which case $N_\Delta((g,h), (g',h')) = 0$ or there is some $j$ such that $c_j \not = c'_j$ in which case $N_\Delta((g,h), (g',h')) \leq \min \{ n_j, n'_j \}$.
Crudely, $N_{\Delta}$ measures how much the triangle corresponding to $(g,h)$ agrees with the triangle corresponding to $(g',h')$ arround the ``centre'' of the triangle; see Figure \ref{fig:Figure2}. 
To illustrate $N_\Delta$ we will give an example for $\Delta$ the trivial decomposition.
\begin{exmp}
Let $\Delta$ be the trivial decomposition and let $g,h,i \in F$ be such that $ghi$ has no cancellation and assume that $g$ is not-trivial.
Then we claim that $N_\Delta((gh,i), (h,i)) = |h|$, where $|h|$ is the word-length of $h$.
To see this observe that the $r$-part of the $\Delta$ triangles of $(gh,i)$ and $(h,i)$ agrees (it's both $(\emptyset, \emptyset, \emptyset)$.
Moreover, the $c$-part of the $\Delta$-triangles $(gh,i)$ and $(h,i)$ is
$$
(\Delta(h)^{-1} \cdot \Delta(g)^{-1}, \emptyset, \Delta(i)) = (\underline{c}_1, \underline{c}_2, \underline{c}_3)
$$
and
$$
(\Delta(h)^{-1}, \emptyset, \Delta(i)) = (\underline{c}'_1, \underline{c}'_2, \underline{c}'_3).
$$
We see that $\underline{c}_2 = \underline{c}_2'$ and $\underline{c}_3 = \underline{c}_3'$ but 
$\underline{c}_1 \not = \underline{c}_1'$. Observe that the length of $\underline{c}_1$ is $|h|+|g|$ and the length of $\underline{c}'_1$ is $|h|$. Moreover, $c_{1,k} = c'_{1,k}$ for every $k \leq |h|$.
This shows that indeed $N_\Delta((gh,i), (h,i)) = |h|$.
\end{exmp}

 \begin{figure}
  \centering

  \subfloat{\includegraphics[width=0.3\textwidth]{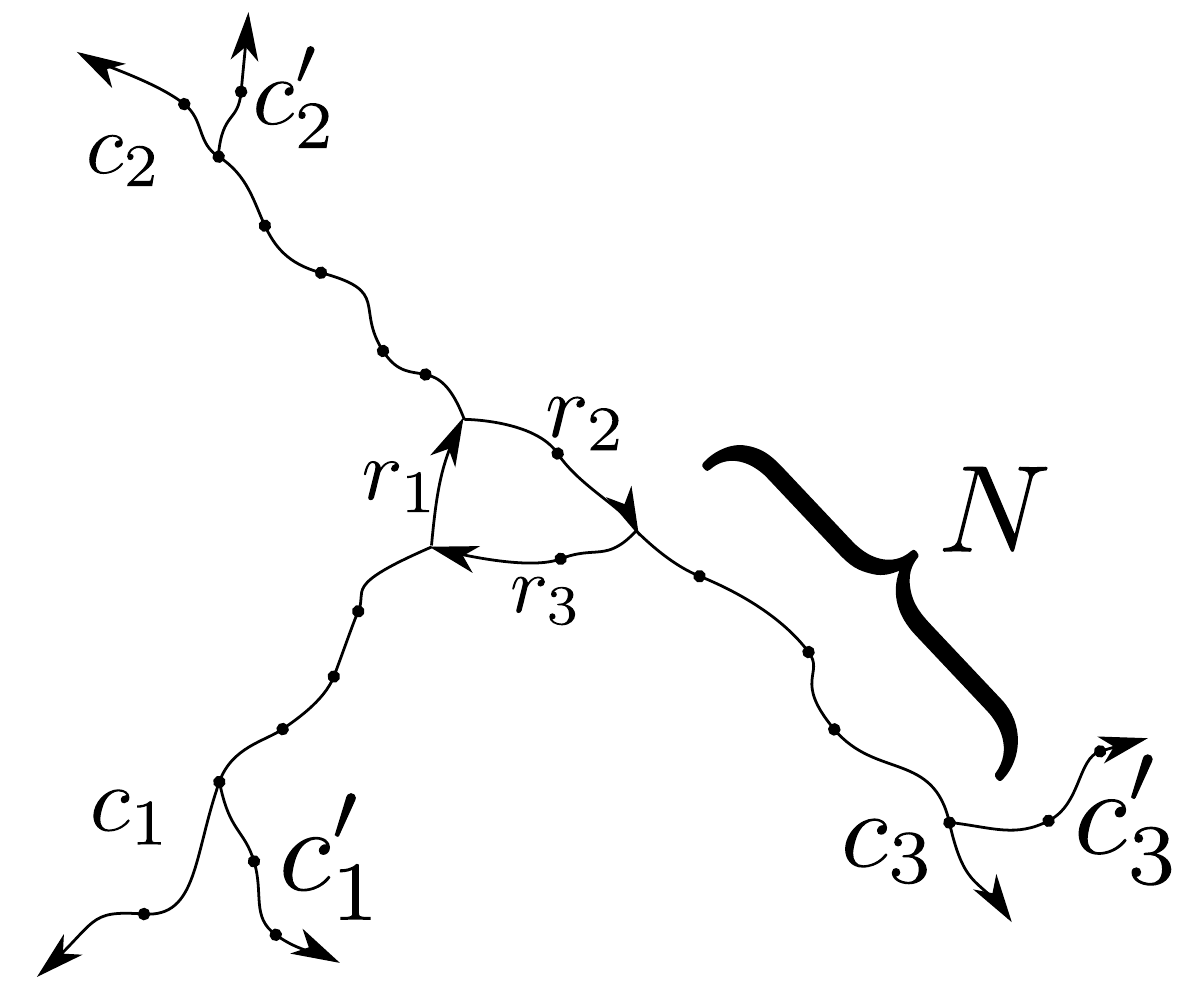}}

  \caption{The $\Delta$-triangle for $(g,h)$ vs. the $\Delta$-triangle for $(g',h')$ and $N = N_{\Delta}((g,h), (g',h'))$} \label{fig:Figure2}
\end{figure}

\begin{defn} \label{def:continuous}
Let $\Delta$ be a decomposition of $F$ and let $N_\Delta$ be as above.
A quasimorphism $\phi$ is called $\Delta$-continuous if $\phi$ is symmetric (i.e.\ $\phi(g^{-1}) = -\phi(g)$ for every $g \in G$) and $\omega = \delta^1 \phi$
satisfies that there is a constant $S_{\omega, \Delta} > 0$ and a non-negative summable sequence $(s_j)_{j \in \N}$ with $\sum_{j=0}^\infty s_j = S_{\omega, \Delta}$ such that for all 
$(g,h), (g',h') \in F^2$ we have that either $(g,h) = (g',h')$ or, 
$$
|\omega(g,h)-\omega(g',h')| \leq s_{N}
$$
where $N = N_\Delta((g,h),(g',h'))$.
In this case we call $\omega$ $\Delta$-continuous as well.
\end{defn}

Crudely, a cocycle $\omega$ is $\Delta$-continuous if its values depend mostly on the parts of the decomposition which lies 
close to the centre of the triangle $g,h,h^{-1} g^{-1}$.

Many quasimorphisms are $\Delta$-continuous as the following proposition shows.
\begin{prop} \label{prop: continuous quasimorphsms}
Let $\Delta$ be a decomposition of $F$. 
\begin{enumerate}
\item  \label{item: decomposable is continuous} Every $\Delta$-decomposable quasimorphism is $\Delta$-continuous.
\item \label{item:brooks is continuous} Every Brooks quasimorphism $\phi \col F \to \R$ is $\Delta$-continuous.
\end{enumerate}
\end{prop}

\begin{proof}
To see $(1)$ observe that the proof of Proposition \ref{prop: decompositions define quasimorphisms} 
shows that $\delta^1 \phi(g,h)$ does not depend on the $c$-part of the $\Delta$-triangle of $(g,h)$. Hence if $N_\Delta((g,h),(g',h')) \geq 1$, then $\delta^1 \phi(g,h) = \delta^1 \phi(g',h')$.

 For (\ref{item:brooks is continuous}), suppose that  $\delta^1 \phi$ is a bounded cocycle induced by a Brooks quasimorphism  $\phi$ on a word $w$ and suppose that the length of $w$ is $m$. The value of the Brooks cocycle  $\delta^1 \phi(g,h)$ just depends on the $m$-neighbourhood of the midpoint of the tripod with endpoints $e, g, gh$ in the Cayley graph.
 Hence, whenever $N_\Delta((g,h),(g',h')) \geq m$ we have that
 $\delta^1 \phi(g,h) = \delta^1 \phi(g',h')$. Note that this implies that Brooks quasimorphisms are $\Delta$-continuous for \emph{any} decomposition $\Delta$.
\end{proof}

\subsection{Triangles and quadrangles in a tree} \label{subsec:triangles and quadrangles}

Let $g,h \in F$. It is easy to see that there are unique elements $t_1,t_2,d \in F$ such that $g = t_1^{-1} d$ and $h = d^{-1} t_2$ as reduced words and that $t_1$, $t_2$ and $d$ are the paths of the tripod with endpoints $1, g, gh$ in the Cayley graph of $F$.
We will call $d$ the \emph{common 2-path} of $(g,h)$.

For three elements $g,h,i \in F$ there are three different cases how the geodesics between the points $1, g, gh, ghi$ in the Cayley graph of $F$ can be aligned. See Figure \ref{fig:Figure4}.
\begin{enumerate}
\item (Figure \ref{fig:Figure4_1}): There are elements $t_1, \ldots, t_5$ such that
 $g = t_1 t_2$, $h = t_2^{-1} t_3 t_4$, $i = t_4^{-1} t_5$ as reduced words.
 \item (Figure \ref{fig:Figure4_2}): There are elements $t_1, \ldots, t_5$ such that
 $g = t_1 t_2 t_3$, $h = t_3^{-1} t_4$, $i = t_4^{-1} t_2^{-1} t_5$ as reduced words.
 \item (Figure \ref{fig:Figure4_3}): There are elements $t_1, \ldots, t_4$ and $c$ such that
$g = t_1^{-1} c t_2$, $h = t_2^{-1} c^{-1} t_3$, $i = t_3^{-1} c t_4$ as reduced words.
\end{enumerate}
We will say that the \emph{common-3-path} of $(g,h,i)$ is empty in the first two cases and $c$ in the third case.
 \begin{figure}
  \centering

  \subfloat[]{\includegraphics[width=0.3\textwidth]{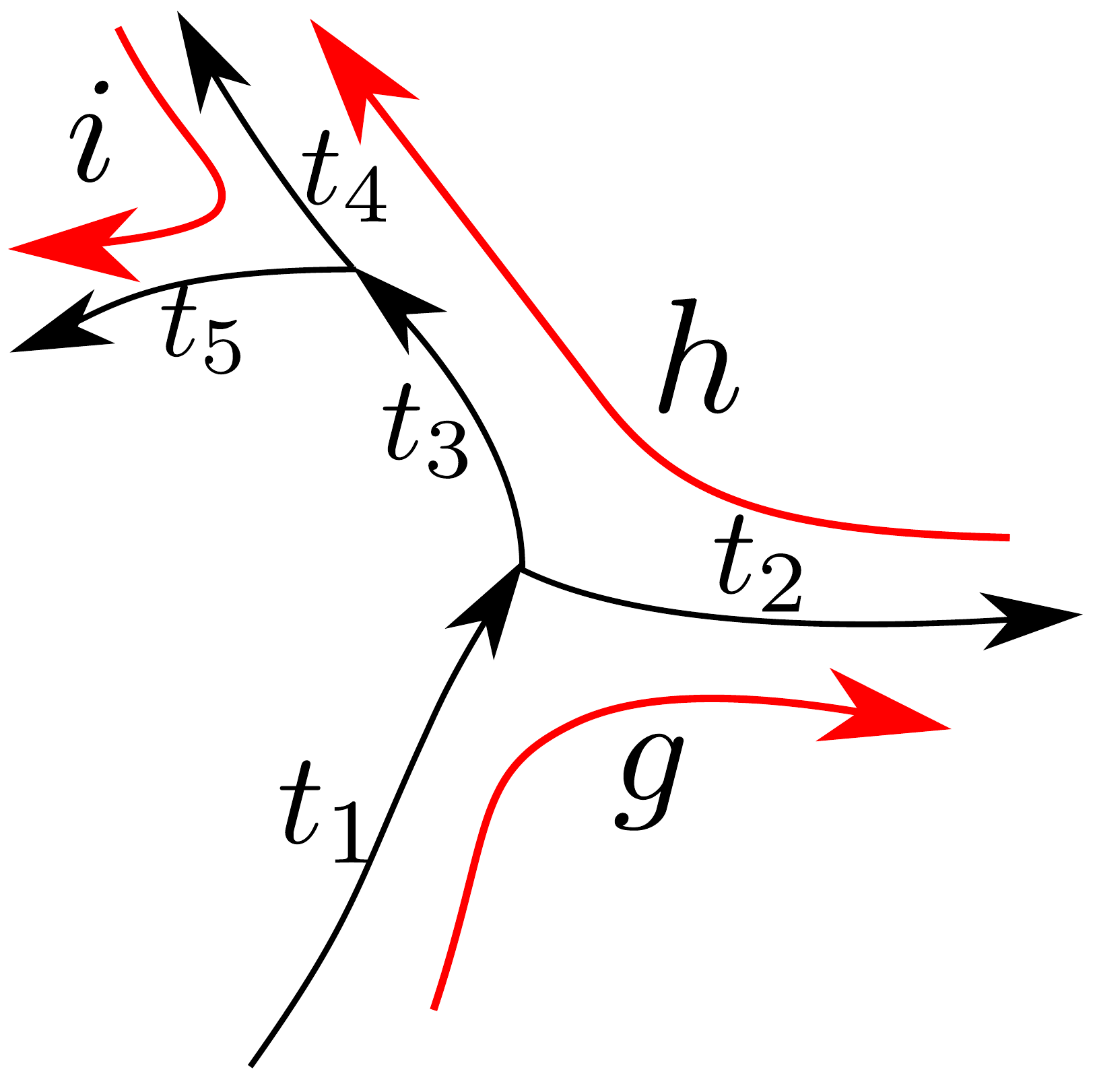} \label{fig:Figure4_1}} 
  \hfill
  \subfloat[]{\includegraphics[width=0.3\textwidth]{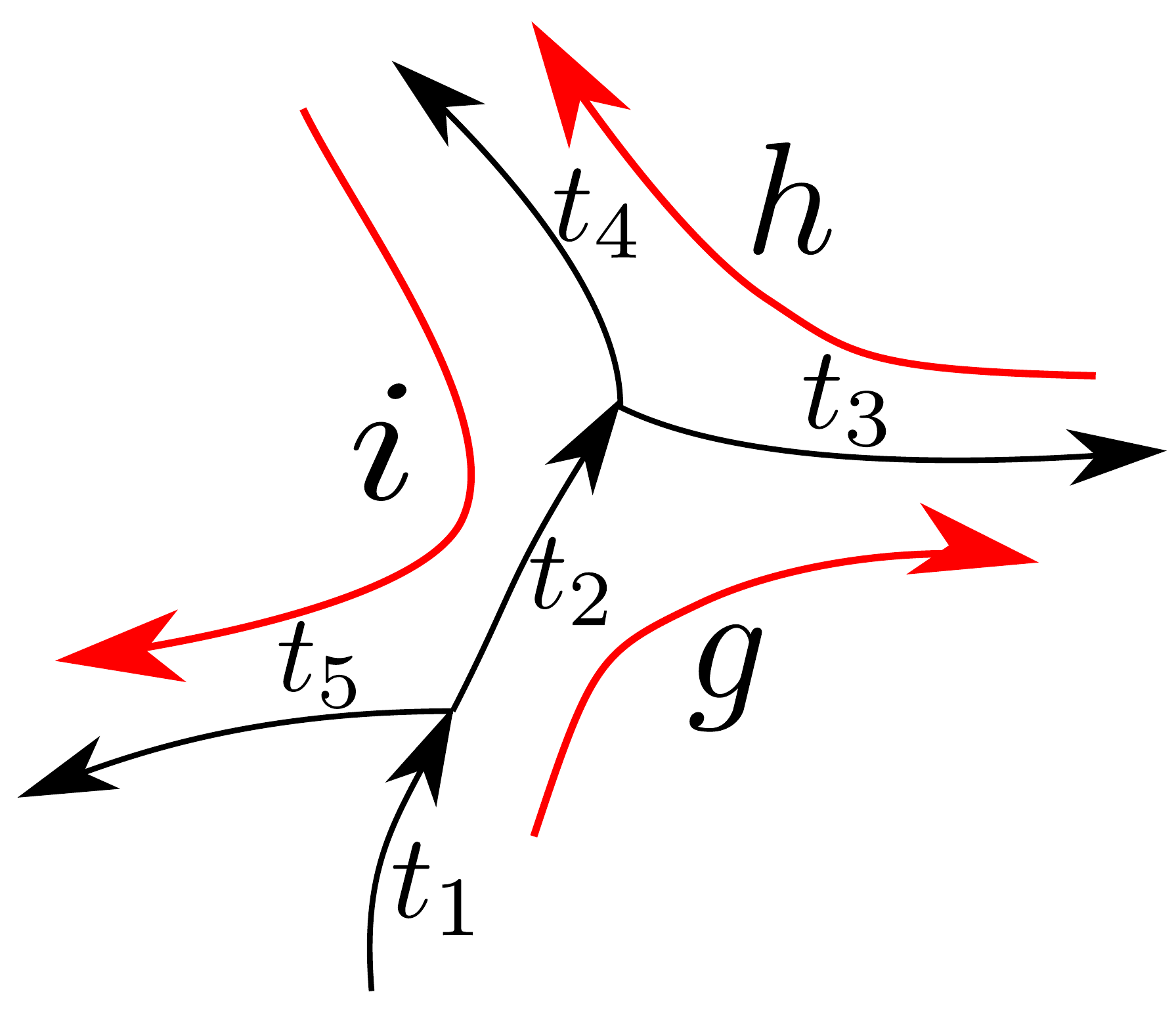} \label{fig:Figure4_2}}
  \hfill
  \subfloat[]{\includegraphics[width=0.3\textwidth]{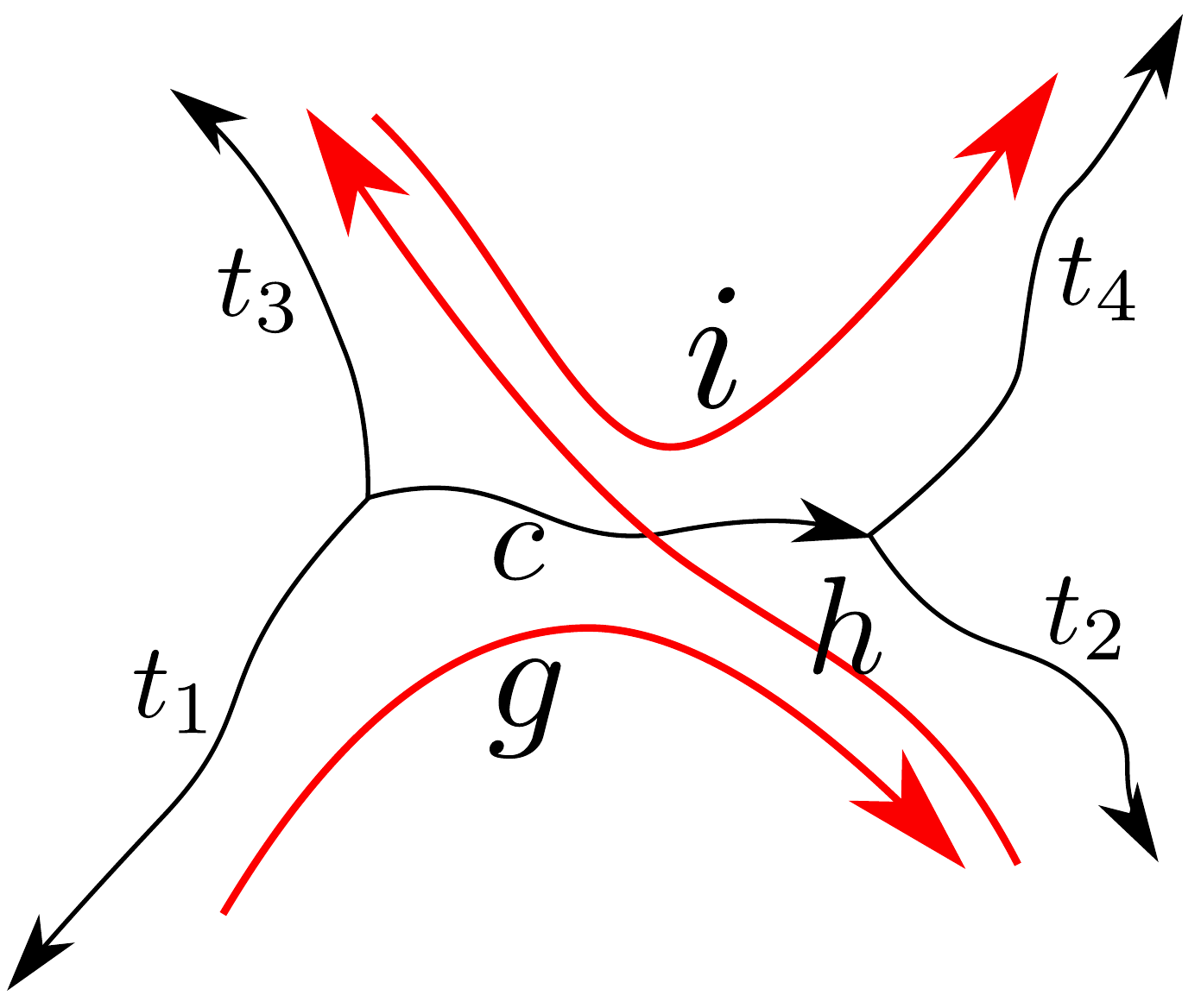} \label{fig:Figure4_3}}
  
\label{Figure4.2}
  \caption{Different cases how $g$, $h$ and $i$ are aligned} \label{fig:Figure4}
\end{figure}

\section{Constructing the bounded primitive} \label{sec:the technical theorem}

Recall that $F$ is a non-abelian group and let $\Delta$ be a decomposition of $F$; see Definition \ref{def:decomposition}. Moreover, let $\phi \col F \to \R$ be a $\Delta$-decomposable quasimorphism (see Definition \ref{defn:decomposable quasimorphism}) and let $\omega \in C^2_b(F, \R)$ be a $\Delta$-continuous symmetric $2$-cocycle (see Definition \ref{def:continuous}).
We define the map $\zeta \in C^3(F,\R)$ by setting
$$
\zeta \col (g,h,i) \mapsto \sum_{j=1}^{k} \phi(g_j) \omega(g_{j+1} \cdots g_k h, i)
$$
for $\Delta(g) = (g_1, \ldots, g_k)$. Moreover, define the maps $\eta, \gamma \in C^2(F,\R)$ by setting
\begin{itemize}
\item $\eta \col (g,h) \mapsto \zeta(g,1,h)$ and
\item $\gamma \col (g,h) \mapsto \frac{1}{2}\Big( \zeta(d,d^{-1},d) + \zeta(d^{-1},1,d) \Big)$ for $d$ the common $2$-path of $(g,h)$; see Subsection \ref{subsec:triangles and quadrangles}.
\end{itemize}

We will show the following theorem:
\begin{theorem} \label{thm:technical}
Let $\phi$ be a $\Delta$-decomposable quasimorphism and let $\omega$ be a symmetric, $\Delta$-continuous $2$-cocycle. Moreover, let $\gamma$ and $\eta$ be as above.
Then $\beta \in C^3(F, \R)$ defined by setting
$$
\beta \col (g,h,i) \mapsto \phi(g) \omega(h,i) + \delta^2 \gamma(g,h,i) + \delta^2 \eta(g,h,i)
$$
is bounded, i.e.\ $\beta \in C^3_b(F, \R)$.
\end{theorem}

We will see in Subsection \ref{subsec:theorems a and b} that $\beta$ will be the bounded primitive for the cup products studied in this paper.
Before we prove this theorem in Subsection \ref{subsec:proof of theorem technical}, we will give an idea of the proof in Subsection \ref{subsec:idea of technical theorem}.
This will be inspired by a construction of coboundaries to $3$-cocycles which we recall in Subsection \ref{subsec:algebraic proof of free group trivial homology}.

\subsection{Idea of the proof of Theorems \ref{thm:main} and \ref{thm:technical}} \label{subsec:idea of technical theorem}
Theorem \ref{thm:main} 
states that $[\delta^1 \phi \smile \omega] = 0$ in $\Hrm^4_b(F, \R)$ for $\phi$ a $\Delta$-decomposable quasimorphism and $\omega$ a $\Delta$-continuous cocycle. Equivalently, there is a bounded primitive of the map
 $\delta^3 \tau \in C^3(F,\R)$, where $\tau$ is given by $\tau \col (g,h,i) \mapsto \phi(g) \omega(h,i)$ since $\delta^3 \tau = \delta^1 \phi \smile \omega$.
Note that $\tau$ is a priori not an interesting function for bounded cohomology: It is neither bounded nor is it a cocycle.

Recall that a map $\alpha \in C^3(F, \R)$ satisfies the cocycle condition if and only if for all $g,g',h,i \in F$ we have that
$$
\delta^3 \alpha(g, g', h, i) = 0.
$$
As $\Hrm^3(F, \R) = 0$, we know that there is some $\epsilon \in C^2 (F, \R)$ such that $\delta^2 \epsilon = \alpha$. We will give a purely algebraic construction of such an $\epsilon$ in terms of $\alpha$, provided $\alpha$ satisfies certain weak conditions stated in Subsection \ref{subsec:algebraic proof of free group trivial homology}, Equation \ref{equ:alternating_condition}. 

Observe that $\tau$ does not satisfy the cocycle condition for \emph{all} $g,g',h,i \in F$. However, $\tau$ satisfies the cocycle condition in certain cases:
 Proposition \ref{prop: decompositions define quasimorphisms} shows that if $g, g' \in F$ satisfy that $\Delta(g \cdot g') = (\Delta(g)) \cdot (\Delta(g'))$ then
$$
\delta^3 \tau(g, g', h, i) = 0
$$
for all $h,i \in F$.
Following the techniques of Subsection \ref{subsec:algebraic proof of free group trivial homology} we will construct an $\epsilon \in C^2(F, \R)$ such that $\delta^2 \epsilon$ is \emph{boundedly close} to $\tau$. This is, such that the map $\beta = \tau - \delta^2 \epsilon$ is bounded, i.e.\ $\beta \in C_b^3(F, \R)$.
This will imply that
$$
\delta^3 \beta = \delta^3 \tau - \delta^3 \delta^2 \epsilon = \delta^1 \phi \smile \omega
$$
and hence the cup product has a bounded primitive and is trivial in bounded cohomology.

\subsection{Constructing $2$-coboundaries from $3$-cocycles} \label{subsec:algebraic proof of free group trivial homology}
Let $\alpha \in C^3(F, \R)$ be a $3$-cocycle i.e.\ a map such that $\delta^3 \alpha = 0$.
We will show how to construct a map $\epsilon \in C^2(F, \R)$ such that $\delta^2 \epsilon = \alpha$.
We emphasize that this subsection just motivates the strategy of the proof of Theorem \ref{thm:technical}. This theorem will be proved in detail in Subsection \ref{subsec:proof of theorem technical} and the proof can be understood without reading this subsection. In both subsections, the $\eta$ and the $\zeta$ term will play analogous r\^{o}le.

To simplify our calculations we will assume that $\alpha$ is a cocycle and moreover satisfies
\begin{eqnarray} \label{equ:alternating_condition}
\alpha(g,h,1) = \alpha(g,1,h) = \alpha(1,g,h) = \alpha(g,g^{-1},h) = 0 \mbox{ for all $g, h \in F$}.
\end{eqnarray} 
We note that \emph{alternating cochains} in the sense of Subsection 4.10 of \cite{frigerio} satisfy (\ref{equ:alternating_condition}) and that such maps may be used to fully compute $\Hrm^3(F, \R)$.

Let $\alpha$ be as above and recall that the cocycle condition implies that for all $g, g', h, i \in F$ we have that
\begin{equation} \label{equ:cocycle cond alpha}
\delta^3 \alpha(g,g',h,i) = \alpha(g',h,i) - \alpha(g g',h,i) + \alpha(g,g'h,i) - \alpha(g,g',hi) + \alpha(g,g',h) = 0.
\end{equation}

In a first step we see how $\alpha$ may be rewritten as a sum of elements of the form $\alpha(\xtt, g', h')$, where $\xtt$ is a letter and $g', h' \in F$.
Define $\zeta \in C^3(F, \R)$ by setting
$$
\zeta \col (g,g',h) \mapsto \sum_{j=1}^k \alpha(\xtt_j, \xtt_{j+1} \cdots \xtt_k g', h)
$$ 
where $g = \xtt_1 \cdots \xtt_k$ is the reduced word representing $g$.
We claim that
\begin{claim} \label{claim:ordinary cohomology zeta}
Let $\alpha \in C^3(F, \R)$ be a cocycle satisfying (\ref{equ:alternating_condition}). Then
$$
\alpha(g,h,i) =  \zeta(g,h,i) - \zeta(g,1,hi) + \zeta(g,1,h)
$$
for all $g,h,i \in F$.
\end{claim}

\begin{proof}
direct computation.
\end{proof}
Now define $\eta \in C^2(F, \R)$ by setting
$$
\eta \col (g,h) \mapsto \zeta(g,1,h).
$$
We then see that
\begin{eqnarray*}
\alpha(g,h,i) + \delta^2 \eta(g,h,i) &=& \zeta(g,h,i) + \zeta(h,1,i) - \zeta(gh,1,i)
\end{eqnarray*}
for all $g,h,i \in F$.

\begin{claim} \label{claim:ordinary cohomology zeta two piece}
We have that 
$$
\zeta(g,h,i) + \zeta(h,1,i) - \zeta(gh,1,i) = 
\zeta(d,h,i) + \zeta(d^{-1}, d h, i)
$$
for all $g,h,i \in F$, where $d$ is the common $2$-path of $(g,h)$.
\end{claim}

\begin{proof}
We will prove this by an explicit calculation.
Observe that it is immediate that if $u,v \in F$ are such that $uv$ is reduced then 
\begin{equation} \label{equ:zeta classic additive}
\zeta(uv,g',h) = \zeta(u,vg',h) + \zeta(v,g',h).
\end{equation}
Now rewrite $g = t_1^{-1} d$ and $h = d^{-1} t_2$, where $d$ is the common 2-path of $(g,h)$; see Subsection \ref{subsec:triangles and quadrangles}.
Then by (\ref{equ:zeta classic additive}) we see that
\begin{itemize}
\item $\zeta(g,h,i)=\zeta(t_1^{-1}, d h, i) + \zeta(d, h, i)$
\item $\zeta(h,1,i)= \zeta(d^{-1}, d h, i) + \zeta(t_2, 1, i)$ 
\item $\zeta(gh,1,i)=\zeta(t_1^{-1},d h,i) + \zeta(t_2,1,i)$
\end{itemize}
Hence
$$
\zeta(g,h,i) + \zeta(h,1,i) - \zeta(gh,1,i)=\zeta(d,h,i) + \zeta(d^{-1}, d h, i).
$$
\end{proof}

\begin{claim} \label{claim: ordinary cohomology coboundary}
We have that 
$\zeta(g,h,i) + \zeta(h,1,i) - \zeta(gh,1,i) = 0$ and hence that $\alpha(g,h,i) = \delta^2 \epsilon$ for $\epsilon = - \eta$.
\end{claim}

\begin{proof}
Let $d$ be the common $2$-path of $(g,h)$ as above. Moreover, suppose that $\dtt_1 \cdots \dtt_l$ is the word representing $d$.
By the previous claim,
$\zeta(g,h,i) + \zeta(h,1,i) - \zeta(gh,1,i) = 
\zeta(d,h,i) + \zeta(d^{-1}, d h, i)$.
We calculate
$$
 \zeta(d,h,i) + \zeta(d^{-1}, d h, i) =   \sum_{j=1}^k \Big( \alpha(\dtt_j, \dtt_{j+1} \cdots \dtt_l h, i) + \alpha(\dtt_j^{-1}, \dtt_j \cdots \dtt_l h, i) \Big).
$$
By evaluating $\delta^3 \alpha(\dtt_j, \dtt_j^{-1}, 
\dtt_j \cdots \dtt_l h, i)$ using property (\ref{equ:alternating_condition}) we have that
$\alpha(\dtt_j^{-1}, \dtt_j \cdots \dtt_l h, i) + \alpha(\dtt_j, \dtt_{j+1} \cdots \dtt_l h, i)=0$.
\end{proof}

Together with Claim \ref{claim:ordinary cohomology zeta two piece} the previous claim implies that $\alpha + \delta^2 \eta = \alpha - \delta^2 \epsilon = 0$.

\subsection{Proof of Theorem \ref{thm:technical}}  \label{subsec:proof of theorem technical}

Let $\Delta$ be a decomposition of $F$ (Definition \ref{def:decomposition}), let $\phi$ be a $\Delta$-decomposable quasimorphism (Definition \ref{defn:decomposable quasimorphism}) and let $\omega$ be a $\Delta$-continuous cocycle (Definition \ref{def:continuous}).
See the previous subsection for a brief discussion on the classical computations that inspired our construction here.
Analogously to Claim \ref{claim:ordinary cohomology zeta}, we will first rewrite the function $(g,h,i) \mapsto \phi(g) \omega(h,i)$  as sum of terms $\phi(g_j) \omega( g', h')$ where $g_j$ will be a piece of a fixed decomposition $\Delta$.
We will construct a map $\epsilon \in C^2(F, \R)$ such that $\delta^2 \epsilon$ is boundedly close to $(g,h,i) \mapsto \phi(g) \omega(h,i)$ by ``treating'' this function as a cocycle on the pieces of $\Delta$ and then performing the calculations of Subsection \ref{subsec:algebraic proof of free group trivial homology}.
For this, define $\zeta \in C^3(F, \R)$ by setting
$$
\zeta(g,g',h) := \sum_{j=1}^{k} \phi(g_j) \omega(g_{j+1} \cdots g_k g', h)
$$
for $\Delta(g) = (g_1, \ldots, g_k)$.
Analogous to Claim \ref{claim:ordinary cohomology zeta} we show:
\begin{prop} \label{prop:rewrite_quasicocycle}
The term $\phi(g) \omega(h,i)$ is equal to 
$$
\zeta(g,h,i) - \zeta(g,1,hi) + \zeta(g,1,h)
$$
for $\zeta \in C^3(F,\R)$ are as above. 
\end{prop}

\begin{proof}
Let $\Delta(g) = (g_1, \ldots, g_k)$. 
Observe that for all $j \in \{1, \ldots, k\}$ by Proposition \ref{prop: decompositions define quasimorphisms} we have that
\begin{eqnarray*}
 0= \delta^1 \phi(g_j, g_{j+1} \cdots g_k) \omega( h, i) &=& \phi(g_{j+1} \cdots g_k) \omega( h, i) - \phi(g_j g_{j+1} \cdots g_k) \omega( h, i)
+ \phi(g_j) \omega( g_{j+1} \cdots g_k h, i) +\ldots \\
& & -\phi(g_j) \omega(g_{j+1} \cdots g_k, hi) + \phi(g_j) \omega(g_{j+1} \cdots g_k, h)
\end{eqnarray*}
Rearranging terms we see that
$$
\phi(g_j \cdots g_k) \omega(h,i) - \phi(g_{j+1} \cdots g_k) \omega(h,i) =  \phi(g_j) \omega(g_{j+1} \cdots g_k h,i) - \phi(g_j) \omega(g_{j+1} \cdots g_k ,hi) + \phi(g_j) \omega(g_{j+1} \cdots g_k ,h).
$$
Summing for $j=1, \ldots, k-1$ over both sides
$$
\phi(g_1 \cdots g_k) \omega( h, i) - \phi(g_k) \omega(h, i)  = \sum_{j=1}^{k-1} \Big( \phi(g_j) \omega( g_{j+1} \cdots g_{k} h, i) -  \phi(g_j) \omega(g_{j+1} \cdots g_{k} ,h i) +  \phi(g_j) \omega(g_{j+1} \cdots g_{k}, h) \Big).
$$
As $\omega$ was supposed to be symmetric we have that
 $\omega(1, h) = \omega(1,hi) = 0$ and hence
$$
\phi(g) \omega(h,i) = \zeta(g,h,i) - \zeta(g,1,hi) + \zeta(g, 1, h).
$$
\end{proof}

As in Subsection \ref{subsec:algebraic proof of free group trivial homology} define $\eta \in C^2(F, \R)$ by setting
$$
\eta \col (g,h) \mapsto \zeta(g,1,h)
$$
and note that
\begin{align*}
\delta^2 \eta(g,h,i) &= \eta(h,i) - \eta(gh,i) + \eta(g,hi) - \eta(g,h) \\
&= \zeta(h,1,i) - \zeta(gh,1,i) + \zeta(g,1,hi) - \zeta(g,1,h).
\end{align*}
Using Proposition \ref{prop:rewrite_quasicocycle} we see that
$$
\phi(g)\omega(h,i)  + \delta^2 \eta(g,h,i) 
$$
is equal to
$$
\zeta(g,h,i) + \zeta(h,1,i) - \zeta(gh,1,i).
$$

We will need the following properties of $\zeta$. 

\begin{prop} \label{prop:properties of zeta}
The function $\zeta$ defined as above has the following properties.
\begin{enumerate}
\item \label{item:splitting of zeta} If $u_1, u_2,v \in F$ are such that $u_1 u_2$ is reduced then $\zeta(u_1 u_2,1,v) - \zeta(u_1,u_2,v) - \zeta(u_2,1,v)$ 
is uniformly bounded.
\item \label{item:two sum zeta}
Let $u_1, u_2, u_3, u_4 \in F$ be elements such that $u_1 u_2$, $u_2 u_3$ and $u_2 u_4$ are reduced and $u_3$ and $u_4$ do not start with the same letter.
Then
$$
\zeta(u_1^{-1}, u_1 u_2 u_3, u_3^{-1} u_4)
+
\zeta(u_1, u_2 u_3, u_3^{-1} u_4)
$$
is uniformly bounded.
\item \label{item:continuity of zeta}
Let $u, v_1, v_2 \in F$ such that $v_1 u v_2$ is reduced. Then
\begin{enumerate}
\item \label{subitem:continuity1} $
\zeta(u, u^{-1} v_1^{-1}, v_1 u v_2) - \zeta(u, u^{-1}, u)
$
and
\item \label{subitem:continuity2} $
\zeta(u^{-1}, v_1^{-1}, v_1 u v_2) - \zeta(u^{-1}, 1, u)
$
\end{enumerate}
are uniformly bounded.
\end{enumerate}
\end{prop}

\begin{proof}
In the proof of items (\ref{item:splitting of zeta})-(\ref{item:continuity of zeta}) we will frequently use the following claim:
\begin{claim} \label{claim:u v no cancellation delta}
Let $u,v_1, v_2 \in F$ be such that $v_1 u v_2$ is reduced, let $\Delta(u) = (u_1, \ldots, u_n)$ and let $R$ be as in Definition \ref{def:decomposition}. Then, there are sequences $(\underline{v}'_1), (\underline{v}'_2)$ such that
\begin{enumerate}
\item \label{item:u right} for every $1 \leq j \leq n-R$ we have that $\Delta(u_j \cdots u_n v_2) = (u_j, \ldots, u_{n-R}) \cdot (\underline{v}_2')$ and
\item \label{item:u left} for every $R \leq j \leq n$ we have that $\Delta(v_1 \cdot u_1 \cdots u_j) = (\underline{v}'_1) \cdot (u_R, \ldots, u_j)$.
\end{enumerate} 
\end{claim}

\begin{proof}
For (\ref{item:u right}) let $(\underline{c}_1), (\underline{c}_2),(\underline{c}_3)$ be the $c$-part of the $\Delta$-triangle of $(u, v_2)$ and let  $(\underline{r}_1), (\underline{r}_2),(\underline{r}_3)$ be the $r$-part of the $\Delta$-triangle of $(u, v_2)$. Then, as $u v_2$ is reduced we see that $(\underline{c}_2)=\emptyset$. Hence
$\Delta(u) = (\underline{c}_1)^{-1} \cdot (\underline{r}_1)$ and $\Delta(uv) = (\underline{c}_1)^{-1} \cdot (\underline{r}_3) \cdot (\underline{c}_3)$. Moreover, observe that the length of $(\underline{r}_1)$ is bounded by $R$. Hence all of $(u_1, \ldots, u_{n-R})$ lie in $(\underline{c}_1)^{-1}$. Comparing $\Delta(uv)$ with $\Delta(u)$ and using that decompositions are infix closed (see Definition \ref{def:decomposition}) yields (\ref{item:u right}). Item (\ref{item:u left}) can be deduced by the same argument. 
\end{proof}

We first show (\ref{item:splitting of zeta}) of Proposition \ref{prop:properties of zeta}. Let $u_1, u_2 \in F$ be such that $u_1 u_2$ is reduced. 
Let the $c$-part of the $\Delta$-triangle of $(u_1, u_2)$ be
$(\underline{c}_1), (\underline{c}_2),(\underline{c}_3)$ and let the $r$-part of the $\Delta$-triangle of $(u_1,u_2)$ be
$(\underline{r}_1), (\underline{r}_2),(\underline{r}_3)$.
As $u_1 u_2$ is reduced, $(\underline{c}_2)$ has to be empty.
Hence 
\begin{itemize}
\item $\Delta(u_1) = ((\underline{c}_1)^{-1} \cdot (\underline{r}_1))$,
\item $\Delta(u_2) = ((\underline{r}_2) \cdot (\underline{c}_3))$ and
\item $\Delta(u_1 u_2) = ((\underline{c}_1)^{-1} \cdot (\underline{r}_3)^{-1} \cdot (\underline{c}_3))$.
\end{itemize} 
Suppose that $(\underline{c}_i)=(c_{i,1}, \ldots c_{i,n_i})$ and that $(\underline{r}_i)=(r_{i,1}, \ldots r_{i,m_i})$ for $i=1,2,3$. Then
\begin{eqnarray*}
\zeta(u_1 u_2,1,v) &=& \sum_{j=1}^{n_1} \phi(c_{1,j}^{-1}) \omega(c_{1,j-1}^{-1} \cdots c_{1,1}^{-1} \bar{r}_3^{-1} \bar{c}_3, v) +
 \sum_{j=1}^{m_3} \phi(r_{3,j}^{-1}) \omega(r_{3,j-1}^{-1} \cdots r_{3,1}^{-1} \bar{c}_3, v) + 
 \\
 & & + \sum_{j=1}^{n_3} \phi(c_{3,j}) \omega(c_{3,j+1}\cdots c_{3,n_3} , v) 
\\
\zeta(u_1,u_2,v) &=& \sum_{j=1}^{n_1} \phi(c_{1,j}^{-1}) \omega(c_{1,j-1}^{-1} \cdots c_{1,1}^{-1} \bar{r}_1 u_2, v) +
\sum_{j=1}^{m_1} \phi(r_{1,j}) \omega(r_{1,j+1}\cdots r_{1,n_1} u_2 , v)
 \\
 \zeta(u_2,1,v) &=& \sum_{j=1}^{m_2} \phi(r_{2,j}) \omega(r_{2,j+1}\cdots r_{2,m_2} \bar{c}_3 , v)
 + \sum_{j=1}^{n_3} \phi(c_{3,j}) \omega(c_{3,j+1}\cdots c_{3,n_3} , v)
\end{eqnarray*}
and hence
\begin{eqnarray*}
\zeta(u_1 u_2,1,v) - \zeta(u_1,u_2,v) - \zeta(u_2,1,v) &=&  \sum_{j=1}^{m_3} \phi(r_{3,j}^{-1}) \omega(r_{3,j-1}^{-1} \cdots r_{3,1}^{-1} \bar{c}_3, v)  \\
& & - \sum_{j=1}^{m_1} \phi(r_{1,j}) \omega(r_{1,j+1}\cdots r_{1,n_1} u_2 , v)  \\
& & -\sum_{j=1}^{m_2} \phi(r_{2,j}) \omega(r_{2,j+1}\cdots r_{2,m_2} \bar{c}_3 , v)
\end{eqnarray*}
which is indeed uniformly bounded, as $m_1,m_2,m_3 \leq R$ (see Definition \ref{def:decomposition}). Since $\phi$ is $\Delta$-decomposable, $\phi$ is uniformly bounded on pieces and as $\omega$ is a bounded function.

To see (\ref{item:two sum zeta}), let $u_1, u_2, u_3, u_4$ be as in the proposition and suppose that $\Delta(u_1) = (u_{1,1}, \ldots, u_{1,n})$. 

\begin{claim} \label{claim:other part}
We have that the $r$-part of the $\Delta$-triangles of $(u_{1,j} \cdots u_{1,n} u_2 u_3, u_3^{-1} u_4 )$ are the same for any $j \leq n-R$ and that the $c$-part of $(u_{1,j} \cdots u_{1,n} u_2 u_3, u_3^{-1} u_4 )$ is $(\underline{c}'_1) \cdot (u_{1,n}^{-1}, \cdots, u_{1,j}^{-1}), (\underline{c}'_2), (\underline{c}'_3)$ for appropriate sequences $(\underline{c}'_1), (\underline{c}'_2), (\underline{c}'_3)$. In particular there is a $C \in \N$ such that 
$$
N_\Delta ((u_{1,j} \cdots u_{1,n} u_2 u_3, u_3^{-1} u_4 ), (u_{1,j+1} \cdots u_{1,n} u_2 u_3, u_3^{-1} u_4 )) = j + C
$$
for all $j \leq n-R$.
\end{claim}
\begin{proof}
It follows by comparing the sequences 
$\Delta(u_{1,j} \cdots u_{1,n} u_2 u_3)$ and 
$\Delta(u_{1,j} \cdots u_{1,n} u_2 u_4)$ using Claim \ref{claim:u v no cancellation delta}.
\end{proof}

For (\ref{item:two sum zeta}) of Proposition \ref{prop:properties of zeta}, we calculate
\begin{eqnarray*}
\zeta(u_1^{-1}, u_1 u_2 u_3, u_3^{-1} u_4)+
\zeta(u_1, u_2 u_3, u_3^{-1} u_4) &=& \sum_{j=1}^n \phi(u_{1,j}^{-1}) \omega(u_{1,j} \cdots u_{1,n} u_2 u_3, u_3^{-1} u_4) \ldots \\
& & + \sum_{j=1}^n \phi(u_{1,j}) \omega(u_{1,j+1} \cdots u_{1,n} u_2 u_3, u_3^{-1} u_4) \\
&=& \sum_{j=1}^n \phi(u_{1,j}) \Big( \omega(u_{1,j+1} \cdots u_{1,n} u_2 u_3, u_3^{-1} u_4) \ldots  \\
& &- \omega(u_{1,j} \cdots u_{1,n} u_2 u_3, u_3^{-1} u_4) \Big).
\end{eqnarray*}
Hence we conclude that $\zeta(u_1^{-1}, u_1 u_2 u_3, u_3^{-1} u_4)+
\zeta(u_1, u_2 u_3, u_3^{-1} u_4)$ is uniformly close to
$$
\sum_{j=1}^{n-R} \phi(u_{1,j}) \Big( \omega(u_{1,j+1} \cdots u_{1,n} u_2 u_3, u_3^{-1} u_4) - \omega(u_{1,j} \cdots u_{1,n} u_2 u_3, u_3^{-1} u_4) \Big)
$$
as $R$ just depends on $\Delta$ and $\phi$ is uniformly bounded on pieces.
Now let $(s_j)_{j \in \N}$ be the sequence in Definition \ref{def:continuous}.
 By Claim \ref{claim:other part},
$$
|\omega(u_{1,j+1} \cdots u_{1,n} u_2 u_3, u_3^{-1} u_4) - \omega(u_{1,j} \cdots u_{1,n} u_2 u_3, u_3^{-1} u_4)| < s_{n+C}.
$$
and hence 
$$
\sum_{j=1}^{n-R} |\omega(u_{1,j+1} \cdots u_{1,n} u_2 u_3, u_3^{-1} u_4) - \omega(u_{1,j} \cdots u_{1,n} u_2 u_3, u_3^{-1} u_4)| < S_{\omega, \Delta}.
$$
Putting those estimations together we see that 
$\zeta(u_1^{-1}, u_1 u_2 u_3, u_3^{-1} u_4)+
\zeta(u_1, u_2 u_3, u_3^{-1} u_4)$ is indeed uniformly bounded.

To see (\ref{subitem:continuity1}), let $u, v_1, v_2$ be as in the proposition and suppose that $\Delta(u) = (u_1, \ldots, u_n)$. By Claim \ref{claim:u v no cancellation delta}, we see that for $n-R \leq j$ and $R \leq j$ the $r$-part of the $\Delta$-triangle of $(u_j^{-1} \cdots u_1^{-1} v_1^{-1}, v_1 u v_2)$ is trivial and that there are sequences $(\underline{v}'_1), (\underline{v}'_2)$ such that the $c$-part of the $\Delta$-triangle is
$\emptyset, (u_j^{-1}, \ldots, u_1^{-1}) \cdot (\underline{v}'_1), (u_{j-1}, \ldots u_n) \cdot (\underline{v}'_2)$.
Hence there are integers $C_1,C_2$ such that
$$
N_\Delta \Big((u_j^{-1} \cdots u_1^{-1} v_1^{-1}, v_1 u v_2), (u_j^{-1} \cdots u_1^{-1},u )\Big) \geq \min \{ j-R+C_1, n-j-R+C_2 \}
$$
and hence
$$
\sum_{j=R}^{n-R} |\omega(u_j^{-1} \cdots u_1^{-1} v_1^{-1}, v_1 u v_2)- (u_j^{-1} \cdots u_1^{-1},u )| \leq 2 S_{\omega, \Delta}.
$$
Finally, observe that
$$
\zeta(u, u^{-1} v_1^{-1}, v_1 u v_2) - \zeta(u, u^{-1}, u) = \sum_{j=1}^n \phi(u_j) \Big(\omega(u_j^{-1} \cdots u_1^{-1} v_1^{-1}, v_1 u v_2)- (u_j^{-1} \cdots u_1^{-1},u ) \Big)
$$
and is hence uniformly close to
$$
\sum_{j=R}^{n-R} \phi(u_j) \Big(\omega(u_j^{-1} \cdots u_1^{-1} v_1^{-1}, v_1 u v_2)- (u_j^{-1} \cdots u_1^{-1},u ) \Big)
$$
With the above estimation we hence see that $\zeta(u, u^{-1} v_1^{-1}, v_1 u v_2) - \zeta(u, u^{-1}, u)$ may be uniformly bounded.
The proof of item (\ref{subitem:continuity2}) is analogous to the proof for item (\ref{subitem:continuity1}).
\end{proof}

Analogously to Claim \ref{claim:ordinary cohomology zeta two piece} we show:
\begin{prop}
The term $\phi(g)\omega(h,i) + \delta^2 \eta(g,h,i)$ is uniformly close to
$\zeta(d, h, i) + \zeta(d^{-1}, d h, i)$ where $d$ is the common $2$-path of $(g,h)$.
\end{prop}

\begin{proof}
Let $g,h,i \in F$. Furthermore write $g = t_1^{-1} d$ and $h = d^{-1} t_2$ where $d$ is the common $2$-piece of $(g,h)$. We know that
$
\phi(g)\omega(h,i)  + \delta^2 \eta(g,h,i) 
$
is equal to
$$
\zeta(g,h,i) + \zeta(h,1,i) - \zeta(gh,1,i).
$$
Using Proposition  \ref{prop:properties of zeta}, (\ref{item:splitting of zeta}) we see that
\begin{itemize}
\item $\zeta(g,h,i)$ is uniformly close to $\zeta(t_1^{-1}, t_2, i)+\zeta(d,d^{-1} t_2, i)$,
\item $\zeta(h,1,i)$ is uniformly close to $\zeta(d^{-1}, t_2, i)+\zeta(t_2,1,i)$ and
\item $\zeta(gh,1,i)$ is uniformly close to $\zeta(t_1^{-1},t_2,i) + \zeta(t_2,1,i)$.
\end{itemize}
Combining these estimates we see that $\phi(g)\omega(h,i)  + \delta^2 \eta(g,h,i)$ is uniformly close to
$\zeta(d, d^{-1} t_2, i) + \zeta(d^{-1}, t_2, i)$.
\end{proof}

\begin{prop} \label{prop: tau plus delta eta unif close to theta}
We have that $\phi(g) \omega( h,i) + \delta^2 \eta(g,h,i)$ is uniformly close to 
$$
\zeta(c, c^{-1}, c) + \zeta(c^{-1}, 1, c)
$$
where $c$ is the common $3$-path of $(g,h,i)$.
\end{prop}

\begin{proof}
We consider the three different cases described in Subsection \ref{subsec:triangles and quadrangles} of how three elements $g, h, i \in F$ can be aligned.
\begin{enumerate}
\item[Case A]: There are elements $t_1, \ldots, t_5$ such that
 $g = t_1 t_2$, $h = t_2^{-1} t_3 t_4$, $i = t_4^{-1} t_5$ as reduced words.
Then the common $2$-path of $(g,h)$ is $t_2$. Hence $\phi(g)\omega(h,i)  + \delta^2 \eta(g,h,i)$ is uniformly close to
$$
\zeta(t_2, t_2^{-1} t_3 t_4, t_4^{-1} t_5) + 
\zeta(t_2^{-1}, t_3 t_4, t_4^{-1} t_5).
$$
Using Proposition \ref{prop:properties of zeta}, (\ref{item:two sum zeta}) for $u_1=t_2^{-1}, u_2 = t_3, u_3 = t_4, u_4 = t_5$ we see that in this case $\phi(g)\omega(h,i)  + \delta^2 \eta(g,h,i)$ is uniformly bounded.
 \item[Case B]: There are elements $t_1, \ldots, t_5$ such that
 $g = t_1 t_2 t_3$, $h = t_3^{-1} t_4$, $i = t_4^{-1} t_2^{-1} t_5$ as reduced words.
Then the common $2$-path of $(g,h)$ is $t_3$. 
 Hence $\phi(g)\omega(h,i)  + \delta^2 \eta(g,h,i)$ is uniformly close to
$$
\zeta(t_3, t_3^{-1} t_4, t_4^{-1} t_2^{-1} t_5)
+
\zeta( t_3^{-1} ,t_4, t_4^{-1} t_2^{-1} t_5).
$$
Using Proposition \ref{prop:properties of zeta}, (\ref{item:two sum zeta}) for $u_1 = t_3^{-1}, u_2 = \emptyset, u_3=t_4,  u_4= t_2^{-1} t_5$ we see that in this case, $\phi(g)\omega(h,i)  + \delta^2 \eta(g,h,i)$ is uniformly bounded.
 \item[Case C]: There are elements $t_1, \ldots, t_4$ and $c$ such that
$g = t_1^{-1} c t_2$, $h = t_2^{-1} c^{-1} t_3$, $i = t_3^{-1} c t_4$ as reduced words.
Then the common $2$-path of $(g,h)$ is $c t_2$.
Hence $\phi(g)\omega(h,i)  + \delta^2 \eta(g,h,i)$ is uniformly close to
$$
\zeta(c t_2, t_2^{-1} c^{-1} t_3, t_3^{-1} c t_4)
+
\zeta(t_2^{-1} c^{-1} ,t_3, t_3^{-1} c t_4)
$$
Using Proposition \ref{prop:properties of zeta} (\ref{item:splitting of zeta}) we see that 
\begin{itemize}
\item $\zeta(c t_2, t_2^{-1} c^{-1} t_3, t_3^{-1} c t_4)$ is uniformly close to
$\zeta(c, c^{-1} t_3, t_3^{-1} c t_4) + \zeta(t_2, t_2^{-1} c^{-1} t_3, t_3^{-1} c t_4)$ and
\item  $\zeta(t_2^{-1} c^{-1} ,t_3, t_3^{-1} c t_4)$ is uniformly close to
$\zeta(t_2^{-1}, c^{-1} t_3, t_3^{-1} c t_4) + \zeta( c^{-1} ,t_3, t_3^{-1} c t_4)$.
\end{itemize}
Hence $\phi(g)\omega(h,i)  + \delta^2 \eta(g,h,i)$ is uniformly close to
\begin{equation*}
\zeta(c, c^{-1} t_3, t_3^{-1} c t_4) + \zeta( c^{-1} ,t_3, t_3^{-1} c t_4) + \Big( \zeta(t_2, t_2^{-1} c^{-1} t_3, t_3^{-1} c t_4) + \zeta(t_2^{-1}, c^{-1} t_3, t_3^{-1} c t_4) \Big).
\end{equation*}
Using Proposition \ref{prop:properties of zeta} (\ref{item:two sum zeta}) for $u_1 = t_2^{-1}, u_2 = c^{-1}, u_3 = t_3, u_4 = c t_4$ we see that
$\Big( \zeta(t_2, t_2^{-1} c^{-1} t_3, t_3^{-1} c t_4) + \zeta(t_2^{-1}, c^{-1} t_3, t_3^{-1} c t_4) \Big)$
is uniformly bounded.
Using item (\ref{subitem:continuity1}) of the same proposition for $u=c, v_1=t_3^{-1}, v_2 = t_4$ we see that 
$\zeta(c, c^{-1} t_3, t_3^{-1} c t_4)$ is uniformly close to $\zeta(c, c^{-1},  c )$ and by item (\ref{subitem:continuity2}) again for $u=c, v_1 = t_3^{-1}, v_2 = t_4$ we see that
$\zeta( c^{-1} ,t_3, t_3^{-1} c t_4)$  is uniformly close to $\zeta(c^{-1}, 1, c)$.
Putting the above estimations together we see that
$\phi(g)\omega(h,i)  + \delta^2 \eta(g,h,i)$ is uniformly close to $\zeta(c, c^{-1}, c) + \zeta(c^{-1}, 1, c)$.
\end{enumerate}
\end{proof}

\begin{prop} \label{prop:theta symmetric quasimorph}
The map $\theta \col F \to \R$ defined by setting 
$$
\theta \col g \mapsto \zeta(g, g^{-1}, g) + \zeta(g^{-1}, 1, g)
$$
is a symmetric quasimorphism.
\end{prop}

\begin{proof}
We will first show the following claim:
\begin{claim} \label{claim:theta straight additive}
If $v, w \in F$ are such that $v w$ is reduced then $\theta(v w)$ is uniformly close to $\theta(v) + \theta(w)$.
\end{claim}
\begin{proof}
Note that $\theta(v w) =  \zeta(v w, w^{-1} v^{-1}, v w) + \zeta( w^{-1} v^{-1}, 1, v w)$. Using Proposition \ref{prop:properties of zeta} (\ref{item:splitting of zeta}) we see that $\theta(v w)$ is uniformly close to
$$
\zeta(w, w^{-1} v^{-1}, v w) + \zeta(v, v^{-1}, v w) + \zeta( w^{-1} ,v^{-1}, v w) + \zeta( v^{-1}, 1, v w).
$$
By item (\ref{item:continuity of zeta}) of the same proposition we see that
\begin{itemize}
\item $\zeta(w, w^{-1} v^{-1}, v w)$ is uniformly close to $\zeta(w, w^{-1} , w)$, for $u = w, v_1 = v, v_2 = \emptyset$, 
\item $\zeta(v, v^{-1}, v w)$ is uniformly close to $\zeta(v, v^{-1}, v)$, for $u=v, v_1 = \emptyset, v_2 = w$,
\item $\zeta( w^{-1} ,v^{-1}, v w)$ is uniformly close to $\zeta(w^{-1}, 1, w)$ for $u = w, v_1 = v, v_2 = \emptyset$ and
\item $\zeta( v^{-1}, 1, v w)$ is uniformly close $\zeta( v^{-1}, 1, v )$ for $u = v, v_1 = \emptyset, v_2 = w$.
\end{itemize}
Putting things together we see that $\theta(vw)$ is uniformly close to
$$
\Big( \zeta(v, v^{-1}, v) + \zeta( v^{-1}, 1, v ) \Big) + \Big( \zeta(w, w^{-1} , w) +  \zeta(w^{-1}, 1, w) \Big) = \theta(v) + \theta(w).
$$
\end{proof}

\begin{claim} \label{claim:theta symmetric}
The map $\theta \col F \to \R$ is symmetric i.e.\ $\theta(g) = - \theta(g^{-1})$ for all $g \in F$.
\end{claim}

\begin{proof}
We first need two easy properties of $\omega$. Note that $\omega$ is induced by a symmetric quasimorphism, say $\omega = \delta^1 \rho$ for some quasimorphism $\rho \col F \to \R$. We have that for all $u,v \in F$,
\begin{equation} \label{equ:alternating cocycle stuff}
\omega(u, u^{-1} v) = \rho(u) + \rho(u^{-1} v ) - \rho(v) = -\rho(u^{-1}) - \rho(v) + \rho(u^{-1} v) = -\omega(u^{-1},v).
\end{equation}
and
\begin{equation} \label{equ:alternating cocycle stuff 2}
\omega(u,v) = \rho(u) + \rho(v) - \rho(uv) = -\rho(u^{-1})-\rho(v^{-1}) - \rho(v^{-1} u^{-1}) = -\omega(v^{-1},u^{-1}).
\end{equation}
Fix $g \in F$ such that $\Delta(g) = (g_1, \ldots, g_k)$. Recall that in this case $\Delta(g^{-1}) = (g_k^{-1}, \ldots, g_1^{-1})$. Then
\begin{itemize}
\item $
\zeta(g,g^{-1},g) = \sum_{j=1}^k \phi(g_j) \omega(g_j^{-1} \cdots g_1^{-1}, g) 
= \sum_{j=1}^k \phi(g_j) \omega(g_1 \cdots g_j, g_{j+1} \cdots g_k)
$ using (\ref{equ:alternating cocycle stuff}) for $u = g_j^{-1} \cdots g_1^{-1}$ and $v = g_{j+1} \cdots g_k$. Similarly we see that
\item $\zeta(g^{-1},1,g) =  \sum_{j=1}^k \phi(g_j^{-1}) \omega(g_1 \cdots g_{j-1}, g_j \cdots g_k) = - \sum_{j=1}^k \phi(g_j) \omega(g_1 \cdots g_{j-1}, g_j \cdots g_k)$ using that $\phi$ is symmetric,
\item $\zeta(g^{-1},g,g^{-1}) = \sum_{j=1}^k \phi(g_j^{-1}) \omega(g_k^{-1} \cdots g_j^{-1}, g_{j-1}^{-1} \cdots g_1^{-1}) = \sum_{j=1}^k \phi(g_j) \omega(g_1 \cdots g_{j-1}, g_j \cdots g_k)$ where we used that $\phi$ is symmetric and (\ref{equ:alternating cocycle stuff 2}) and
\item $\zeta(g, 1, g^{-1}) = \sum_{j=1}^k \phi(g_j) \omega(g_k^{-1} \cdots g_{j}^{-1}, g_{j-1}^{-1} \cdots g_1^{-1}) = -\sum_{j=1}^k \phi(g_j) \omega(g_1 \cdots g_j, g_{j+1} \cdots g_k)$, where we used once more (\ref{equ:alternating cocycle stuff 2}).
\end{itemize}
We hence see that $\theta(g) + \theta(g^{-1}) = \zeta(g,g^{-1},g) +  \zeta(g^{-1},1,g) + \zeta(g^{-1},g,g^{-1}) + \zeta(g, 1, g^{-1}) = 0$ and $\theta$ is symmetric.
\end{proof}
We can now prove that $\theta$ is a quasimorphism. Let $g, h \in F$ and suppose that $d$ is the common $2$-path of $(g,h)$ i.e. $g = t_1^{-1} d$, $h = d^{-1} t_2$ as reduced words for some appropriate $t_1, t_2 \in F$.
Then, by Claim \ref{claim:theta straight additive} we have that $\theta(g) + \theta(h)$ is uniformly close to
$$
\theta(t_1^{-1}) + \theta(d) + \theta(d^{-1}) + \theta(t_2)
$$
and by Claim \ref{claim:theta symmetric}, $\theta(g) + \theta(h)$ is uniformly close to $\theta(t_1^{-1}) + \theta(t_2)$. By Claim \ref{claim:theta straight additive} again, $\theta(t_1^{-1}) + \theta(t_2)$ is uniformly close to $\theta(t_1^{-1} t_2) = \theta(gh)$. Hence $\theta(g) + \theta(h)$ is uniformly close to $\theta(gh)$ and hence $\theta$ is a quasimorphism.
\end{proof}

We will need the following Lemma:
\begin{lemma} \label{lemma:common-piece-cocycle}
Suppose $\rho \col F \to \R$ is a symmetric quasimorphism.
Define $\kappa \in C^2(F,\R)$ by $\kappa(g,h) = \rho(d)$ where $d$ is the common $2$-path of $(g,h)$.
Then $\delta^2 \kappa(g,h,i)$ is uniformly close to $-2 \rho(c)$ where $c$ is the common $3$-path of $(g,h,i)$.
\end{lemma}

\begin{proof}
We have to evaluate
\[
\delta^2 \kappa(g,h,i) = \kappa(h,i) - \kappa(gh,i) + \kappa(g,hi) - \kappa(g,h).
\]
For what follows we will use the different cases of how $g$, $h$ and $i$ can be aligned in the Cayley graph of $F$ as seen in Figure \ref{fig:Figure4}.
\begin{enumerate}
\item (see Figure \ref{fig:Figure4_1}): In this case there are elements $t_1, \ldots, t_5$ such that
 $g = t_1 t_2$, $h = t_2^{-1} t_3 t_4$, $i = t_4^{-1} t_5$ as reduced words.
It follows that 
\begin{itemize}
\item $t_4$ is the common $2$-path of $(h,i)$,
\item $t_4$ is the common $2$-path of $(gh,i)$,
\item $t_2$  is the common $2$-path of $(g,hi)$ and
\item $t_2$  is the common $2$-path of $(g,h)$.
\end{itemize}
Hence $\delta^2 \kappa(g,h,i) = \rho(t_4) - \rho(t_4) + \rho(t_2) - \rho(t_2) = 0$.
\item (see Figure \ref{fig:Figure4_2}): In this case there are elements $t_1, \ldots, t_5$ such that
 $g = t_1 t_2 t_3$, $h = t_3^{-1} t_4$, $i = t_4^{-1} t_2^{-1} t_5$ as reduced words.
It follows that 
\begin{itemize}
\item $t_4$ is the common $2$-path of $(h,i)$,
\item $t_4 t_2$ is the common $2$-path of $(gh,i)$,
\item $t_2 t_3$  is the common $2$-path of $(g,hi)$ and
\item $t_3$  is the common $2$-path of $(g,h)$.
\end{itemize}
Hence $\delta^2 \kappa(g,h,i) = \rho(t_4) - \rho(t_4 t_2) + \rho(t_2 t_3) - \rho(t_3)$ which is uniformly bounded as $\rho$ is a quasimorphism.
\item (see Figure \ref{fig:Figure4_3}): In this case there are elements $t_1, \ldots, t_4$ and $c$ such that
$g = t_1^{-1} c t_2$, $h = t_2^{-1} c ^{-1} t_3$, $i = t_3^{-1} c t_4$ as reduced words.
It follows that 
\begin{itemize}
\item $c^{-1} t_3$ is the common $2$-path of $(h,i)$,
\item $t_3$ is the common $2$-path of $(gh,i)$,
\item $t_2$  is the common $2$-path of $(g,hi)$ and
\item $c t_2$  is the common $2$-path of $(g,h)$.
\end{itemize}
Hence $\delta^2 \kappa(g,h,i) = \rho(c^{-1} t_3) - \rho(t_3) + \rho(t_2) - \rho(c t_2)$ which is uniformly close to $-2 \rho(c)$.
This shows Lemma \ref{lemma:common-piece-cocycle}.
\end{enumerate}
\end{proof}

Finally, we can prove Theorem \ref{thm:technical}.
By Proposition \ref{prop: tau plus delta eta unif close to theta}, $\phi(g) \omega(h,i) + \delta^2 \eta(g,h,i)$ is uniformly close to $\zeta(c, c^{-1}, c) + \zeta(c^{-1}, 1, c)=\theta(c)$ where $c$ is the common $3$-path of $(g,h,i)$ and $\theta \col F \to \R$ is like in Proposition \ref{prop:theta symmetric quasimorph}.
Define $\gamma \in C^2(F, \R)$ via $\gamma(g,h) = \theta(d)/2$ where $d$ is the common $2$-path of $(g,h)$. Observe that $\rho \col g \mapsto \theta(g)/2$ is a symmetric quasimorphism by Proposition \ref{prop:theta symmetric quasimorph}.  Using Lemma \ref{lemma:common-piece-cocycle}, we see that $\delta^2 \gamma(g,h,i)$ is uniformly close to $-\theta(c)$ where $c$ is the common $3$-path of $(g,h,i)$.
Hence $\phi(g) \omega(h,i) + \delta^2 \eta(g,h,i)+\delta^2 \gamma(g,h,i)$ is uniformly bounded.

\subsection{Proof of Theorems \ref{theorem:brooks and rolli} and \ref{thm:main}} \label{subsec:theorems a and b}

Here we will prove Theorems \ref{theorem:brooks and rolli} and \ref{thm:main} by providing an explicit bounded primitive for the respective cup products.

\begin{reptheorem}{thm:main}
Let $\Delta$ be a decomposition of $F$, let $\phi$ be a $\Delta$-decomposable quasimorphism and let $\psi$ be $\Delta$-continuous. Then $[\delta^1 \phi] \smile [\delta^1 \psi] \in \Hrm_b^4(F, \R)$ is trivial.
The bounded primitive is given by $\beta$, as in Theorem \ref{thm:technical} for $\omega = \delta^1 \psi$.
\end{reptheorem}

\begin{proof}
By Theorem \ref{thm:technical} we know that $\beta$ defined by setting $\beta \col (g,h,i) \mapsto \phi(g)  \delta^1 \psi(h,i) + \delta^2 \eta(g,h,i) + \delta^2 \gamma(g,h,i)$ is bounded, as $\delta^1 \psi(h,i)$ is a symmetric $\Delta$-continuous cocycle.
Then we calculate
$$
\delta^3 \beta(g,h,i,j) = \delta^1 \phi(g,h) \smile \delta^1 \psi(i,j).
$$
Hence $\beta$ is a bounded primitive for the cup product.
\end{proof}

Finally, we can prove Theorem \ref{theorem:brooks and rolli}.
\begin{reptheorem}{theorem:brooks and rolli}
Let $\phi, \psi \col F \to \R$ be two quasimorphisms on a non-abelian free group $F$ where each of $\phi$ and $\psi$ is either Brooks counting quasimorphisms on a non self-overlapping word or quasimorphisms in the sense of Rolli. 
Then $[\delta^1 \phi] \smile [\delta^1 \psi] \in \Hrm^4_b(F,\R)$ is trivial.
\end{reptheorem}

\begin{proof}
First suppose that both $\phi$ and $\psi$ are Brooks quasimorphisms. Suppose that $\phi$ is counting the non-overlapping word $w \in F$. Let $\Delta_w$ be the decomposition described in Example \ref{exmp:brooks decomposition}. By Example \ref{exmp:rolli quasimorphism}, we have that $\phi$ is $\Delta_w$-decomposable.
Moreover, by Proposition \ref{prop: continuous quasimorphsms}, $\psi$ is $\Delta_w$-continuous. We conclude by Theorem \ref{thm:main}.

If not both $\phi$ and $\psi$ are Brooks quasimorphisms then assume without loss of generality that $\phi$ is a quasimorphism in the sense of Rolli and $\psi$ is either a Brooks quasimorphism or a quasimorphism in the sense of Rolli. Let $\Delta_{rolli}$ be the decomposition described in Example \ref{exmp:rolli decomposition}.
Note that $\phi$ is $\Delta_{rolli}$-decomposable. If $\psi$ is a quasimorphism in the sense of Rolli, then $\psi$ is $\Delta_{rolli}$-decomposable and hence $\Delta_{rolli}$-continuous by Proposition \ref{prop: continuous quasimorphsms}. If $\psi$ is a Brooks quasimorphism then by the same proposition we see that $\psi$ is also $\Delta_{rolli}$-continuous. Again we may conclude by applying Theorem \ref{thm:main}.
\end{proof}

\bibliographystyle{alpha}
\bibliography{bib_cup}
\end{document}